\documentclass{dcdsOF}
\usepackage{amsmath}
\usepackage{paralist}
\usepackage[colorlinks=true]{hyperref}
\hypersetup{urlcolor=blue, citecolor=red}

\def\currentvolume{X}
\def\currentissue{X}
\def\currentyear{200X}
\def\currentmonth{XX}
\def\ppages{X--XX}
\def\DOI{10.3934/dcds.2022019}

\newtheorem{theorem}{Theorem}[section]
\newtheorem{corollary}[theorem]{Corollary}

\newtheorem{lemma}[theorem]{Lemma}
\newtheorem{proposition}[theorem]{Proposition}

\newtheorem*{problem}{Problem}
\theoremstyle{definition}
\newtheorem{definition}[theorem]{Definition}
\newtheorem{remark}[theorem]{Remark}

\newcommand{\ep}{\varepsilon}

\newtheoremstyle{TheoremNum}
{\topsep}{\topsep}              
{\itshape}                      
{}                              
{\bfseries}                     
{.}                             
{ }                             
{\thmname{#1}\thmnote{ \bfseries #3}}
\theoremstyle{TheoremNum}
\newtheorem{thmrep}{Theorem}

\newtheoremstyle{TheoremNum}
{\topsep}{\topsep}              
{\itshape}                      
{}                              
{\bfseries}                     
{.}                             
{ }                             
{\thmname{#1}\thmnote{ \bfseries #3}}
\theoremstyle{TheoremNum}
\newtheorem{correp}{Corollary}

\newcommand{\nnorm}[1]{\lvert\!|\!| #1|\!|\!\rvert}
\newcommand{\norm}[1]{\left\vert\left\vert {#1}\right\vert\right\vert}

\title[Polynomial ergodic averages for rings] 
{Polynomial ergodic averages for certain countable ring actions} 

\author[Andrew Best and Andreu Ferr\'{e} Moragues]{}

\subjclass{Primary: 37A44; Secondary: 28D15, 05D10, 37B05.}
\keywords{Joint ergodicity, independent polynomials, field actions, ring actions, locally compact abelian groups, Gowers--Host--Kra seminorms, equidistribution.}

\email{best.221@osu.edu}
\email{ferremoragues.1@osu.edu}

\begin{document}
	\maketitle
	
	\centerline{\scshape Andrew Best}
	\medskip
	{\footnotesize
		\centerline{Beijing Institute of Mathematical Sciences and Applications}
		\centerline{Beijing, China}
	} 
	
	\medskip
	
	\centerline{\scshape Andreu Ferr\'{e} Moragues}
	\medskip
	{\footnotesize
		\centerline{Department of Mathematics}
		\centerline{Nicolaus Copernicus University}
		\centerline{Toru\'{n}, Poland}
	}
	
	\bigskip
	
	\centerline{(Communicated by Alejandro Maass)}

	\begin{abstract}
		A recent result of Frantzikinakis in \cite{frajointerg} establishes sufficient conditions for joint ergodicity in the setting of $\mathbb{Z}$-actions. We generalize this result for actions of second-countable locally compact abelian groups. We obtain two applications of this result. First, we show that, given an ergodic action $(T_n)_{n \in F}$ of a countable field $F$ with characteristic zero on a probability space $(X,\mathcal{B},\mu)$ and a family $\{p_1,\dots,p_k\}$ of independent polynomials, we have
		\[ \lim_{N \to \infty} \frac{1}{|\Phi_N|}\sum_{n \in \Phi_N} T_{p_1(n)}f_1\cdots T_{p_k(n)}f_k\ = \ \prod_{j=1}^k \int_X f_i \ d\mu,\]
		where $f_i \in L^{\infty}(\mu)$, $(\Phi_N)$ is a F\o lner sequence of $(F,+)$, and the convergence takes place in $L^2(\mu)$. This yields corollaries in combinatorics and topological dynamics. Second, we prove that a similar result holds for totally ergodic actions of suitable rings.
	\end{abstract}

	\section{Introduction}
	In \cite{ertau}, Bergelson asked whether, for a totally ergodic\footnote{We say that a measure-preserving system $(X,\mathcal{B},\mu,T)$ is \emph{totally ergodic} if $T^a$ is ergodic for all $a \in \mathbb{N}$.} system $(X,\mathcal{B},\mu,T)$, the following statement holds:
	\begin{itemize}
		\item[(*)] For any family of independent\footnote{Let $R$ be a countable ring. We say that a family of polynomials $\{p_1,\dots,p_k\} \subset R[n]$ is \emph{independent} if for any $(b_1,\ldots, b_k) \in R^k \setminus \{\vec{0}\}$, the polynomial $\sum_{i=1}^k b_ip_i$ is not constant.
		} polynomials $p_1,\ldots, p_k \in \mathbb{Z}[n]$ and any functions $f_1,\ldots, f_k \in L^\infty(\mu)$, one has
		\begin{equation*}
			\lim_{N\to\infty} \frac{1}{N} \sum_{n=1}^N T^{p_1(n)}f_1 \cdots T^{p_k(n)}f_k \ = \ \int_X f_1 \ d\mu \ \cdots \int_X f_k \ d\mu,
		\end{equation*}
		where convergence takes place in $L^2(\mu)$.
	\end{itemize} This question was answered in the affirmative in \cite{frakra}. In fact, (*) holds if and only if $T$ is totally ergodic.
	
	Recently, the ergodic theory of actions by groups other than the integers has garnered some attention---see, for example,
	\cite{blm}, \cite{btz}, \cite{zorinkranich}, \cite{dgijm}, \cite{shalom}, \cite{abb}, and \cite{ab}.
	
	In this spirit, we consider in this article a version of Bergelson's question for more general actions. In order to formulate this question, we need a definition. Let $G$ be a countable abelian group. We say that a sequence $(\Phi_N)$ of nonempty subsets of $G$ is a \emph{F\o lner sequence} in $G$ if, for every $g \in G$, we have
	\[ \lim_{N \to \infty} \frac{|\Phi_N \Delta (g+\Phi_N)|}{|\Phi_N|}=0,\]
	where $\Delta$ denotes the symmetric difference, i.e. $A\Delta B=(A \setminus B) \cup (B \setminus A)$.

	\begin{problem}\label{introquestion}
		Let $(X,\mathcal{B},\mu)$ be a probability space. For which (countable) ring actions $(T_r)_{r \in R}$ on $(X,\mathcal{B},\mu)$ is the following statement \emph{(**)} true?
		\begin{itemize}
			\item[\emph{(**)}] For any family of independent polynomials $p_1,\ldots,p_k \in R[n]$, any F\o lner sequence\footnote{A perceptive reader will notice that we did not mention F\o lner sequences in (*). In fact, (*) is equivalent to a version which also quantifies over all F\o lner sequences of $\mathbb{Z}$, not just $\Phi_N=\{1,\ldots,N\}$.} $(\Phi_N)$, and any functions $f_1,\dots,f_k \in L^{\infty}(\mu)$, one has
			\[ \lim_{N \to \infty} \frac{1}{|\Phi_N|}\sum_{n \in \Phi_N} T_{p_1(n)}f_1\cdots T_{p_k(n)}f_k \ = \ \prod_{j=1}^k \int_X f_j \ d\mu, \]
			where convergence takes place in $L^2(\mu)$.
		\end{itemize}
	\end{problem}
	As in Bergelson's question, we continue to require that the polynomials be independent. This assumption cannot be removed; see Remark~\ref{example}.
	
	By analogy with the case of $\mathbb{Z}$-actions, one would expect total ergodicity to be necessary and sufficient for (**) to hold.
	
	In fact, we will show that, for countable fields of zero characteristic, (**) holds if and only if the action is ergodic, and that, for certain suitable rings, (**) holds if and only if the action is totally ergodic. These terms will be defined in Section~\ref{subsec: intro2}.
	
	Our results can be understood in a common framework that we now describe.
	\subsection{Sufficient conditions for joint ergodicity of sequences}
	Let us review one of the main results of Frantzikinakis in \cite{frajointerg}. Given an ergodic measure-preserving system $(X,\mathcal{B},\mu,T)$, a family of sequences $a_1,\dots,a_{\ell}: \mathbb{N} \to \mathbb{N}$
	is \emph{jointly ergodic for} $(X,\mathcal{B},\mu,T)$ if, for all $f_1,\ldots, f_\ell \in L^\infty(\mu)$,
	\[ \lim_{N \to \infty} \frac{1}{N}\sum_{n=1}^N T^{a_1(n)}f_1\cdots T^{a_{\ell}(n)}f_{\ell} \ = \ \prod_{i=1}^{\ell} \int_X f_i \ d\mu\]
	holds in $L^2(\mu)$.
	
	Note that the aforementioned question of Bergelson in \cite{ertau} can obviously be reformulated in terms of joint ergodicity of sequences given by independent polynomials.
	
	Inspired by an argument of Peluse \cite{peluse} and Peluse and Prendiville \cite{peluseprendiville}, Frantzikinakis shows in \cite{frajointerg} that a family of sequences $a_1,\ldots, a_\ell$ is jointly ergodic for $(X,\mathcal{B},\mu,T)$ provided that two conditions hold:
	\begin{enumerate}
		\item The sequences $a_1,\dots,a_{\ell}$ are \emph{good for equidistribution for the system} $(X,\mathcal{B},\mu,\\T)$, meaning that
		\[ \lim_{N \to \infty} \frac{1}{N}\sum_{n=1}^N e^{2\pi i(a_1(n)\alpha_1+\dots+a_{\ell}(n)\alpha_{\ell})} \ = \ 0\]
		holds for any $\alpha_1,\dots,\alpha_{\ell} \in [0,1)$, not all zero, which are such that there exist non-constant functions $f_j: X \to \mathbb{S}^1$, with $Tf_j=e^{2\pi i \alpha_j}f_j$, $j=1,\dots,\ell$.
		\item The sequences $a_1,\dots,a_{\ell}$ are \emph{good for seminorm\footnote{These seminorms $\nnorm{\cdot}_s$ are the Gowers--Host--Kra seminorms, whose definition and basic properties will be given in Section~\ref{sec: 3.1}.} estimates for the system} $(X,\mathcal{B},\mu,T)$, meaning that there exists $s \in \mathbb{N}$ such that whenever $f_1,\ldots, f_\ell \in L^\infty(\mu)$ satisfy $\nnorm{f_l}_s = 0$ for some $l \in \{1,\dots, \ell\}$, then
\[ \lim_{N \to \infty} \frac{1}{N}\sum_{n=1}^N T^{a_1(n)}f_1\cdots T^{a_{\ell}(n)}f_{\ell}=0\]
holds in $L^2(\mu)$, provided that $f_{l+1},\dots,f_\ell$ are eigenfunctions, i.e., belong to $\{f \in L^\infty(\mu) : |f| = 1 \text{ and } Tf = e^{2\pi i t} f \text{ for some } t \in [0,1)\}$.
		
	\end{enumerate}
	For our purposes, it will be useful to create a far-reaching generalization of Frantzikinakis's result from the setting of $\mathbb{Z}$-actions to the setting of actions of second-countable locally compact abelian groups $G$. We state this generalization as Theorem~\ref{framainthm} below; the precise definitions of the relevant terms we generalize from $\mathbb{Z}$ to $G$ will be left to Sections~\ref{sec: 2} and \ref{sec: 3}.
	
	\begin{theorem}\label{framainthm}
		Let $G$ be a second-countable locally compact abelian group. Let $(\Phi_N)$ be a F\o lner sequence. Let $(X,\mu,(T_g)_{g \in G})$ be an ergodic system and $a_1,\dots,a_{\ell} : G \to G$ be mappings that are good for seminorm estimates and equidistribution for $(X,\mu,(T_g)_{g \in G})$. Then, for any $f_1,\dots,f_{\ell} \in L^{\infty}(\mu)$ we have:
				\begin{equation}\label{mainthmeqintro}
			\lim_{N \to \infty} \mathbb{E}_{n \in \Phi_N} T_{a_1(n)}f_1\cdots T_{a_{\ell}(n)} f_{\ell} \ = \ \int_X f_1 \ d\mu \cdots \int_X f_{\ell} \ d\mu \end{equation}
		in $L^2(\mu)$.
	\end{theorem}
	This result is proved in Section~\ref{sec: 3}. Some intermediate results require relatively substantial changes compared to the case of $\mathbb{Z}$-actions. Some proofs of intermediate results are omitted when their adaptations are routine. As with the proof of Frantzikinakis's theorem, our proof of Theorem~\ref{framainthm} uses only ``soft'' tools from ergodic theory and eschews Host--Kra structure theory (see, for example, \cite[Chapter~21]{hostkra} for a detailed description) and equidistribution results on nilmanifolds.
	
	
	\subsection{Answering the Problem, and related corollaries}\label{subsec: intro2}
	
	We will apply Theorem~\ref{framainthm} to deduce Theorems~\ref{findep} and \ref{rindep}, which provide answers to the Problem above. Thus, we need only verify that the necessary equidistribution results and seminorm estimates hold for sequences given by families of independent polynomials.
	
	
	To facilitate discussion, we recall two definitions. Let $G$ be a countable abelian group. An action $(T_g)_{g \in G}$ on a probability space $(X,
	\mathcal{B},\mu)$ is \emph{ergodic} if any $f \in L^2(\mu)$ with $T_gf=f$ for all $g \in G$ is constant, and \emph{totally ergodic} if every subaction on a finite index subgroup $H\leq G$ is ergodic.
	
	\subsubsection{Countable fields of characteristic zero}
	
	Our first application of Theorem~\ref{framainthm} will be to ergodic actions of countable fields of characteristic zero. One might wonder why we choose to work with fields first; a few comments are in order.
	
	It is well known that fields of characteristic zero do not have additive subgroups of finite index. Indeed, this algebraic property means that ergodic actions of such fields are totally ergodic\footnote{This fact has manifested before; see, e.g., \cite{larick}.}. This coincidence allows us to obtain a positive answer to the Problem for any such ergodic action.

	Only requiring ergodicity (instead of total ergodicity) will turn out to have many interesting combinatorial and topological-dynamical consequences that we cannot obtain otherwise.
	
	We will prove the following theorem in Section~\ref{sec: 4}. Recall that the notion of independent polynomials was defined in footnote 2 above.
	\begin{theorem}\label{findep}
		Let $F$ be a countable field with characteristic zero. Let $(X,\mu,\break (T_n)_{n \in F})$ be an ergodic $F$-system. Let $(\Phi_N)$ be a F\o lner sequence in $(F,+)$, let $k \in \mathbb{N}$, let $p_1,\dots,p_k \in F[x]$ be independent polynomials, and let $f_1,\dots,f_k \in L^{\infty}(\mu)$ be functions. Then
		\[ \lim_{N \to \infty} \frac{1}{|\Phi_N|}\sum_{n\in \Phi_N} T_{p_1(n)}f_1\cdots T_{p_k(n)}f_k \ = \ \prod_{j=1}^k \int_X f_i \ d\mu\]
		in $L^2(\mu)$.
	\end{theorem}
	It is worth mentioning that an alternative approach to proving Theorem~\ref{findep} and Theorem~\ref{rindep} would be to use an appropriate analog of the structure theory of Host--Kra characteristic factors \`a la \cite{frakra}. However, an essential ingredient to this approach is a suitable description of said factors, which is not yet available in this generality. Obtaining such a description remains an open problem.
	
	An application of the Furstenberg correspondence principle (see Section~\ref{sec: 4} for details) allows us to deduce the following corollary from Theorem~\ref{findep}. We first recall a notion of largeness. Let $G$ be a countable abelian group. Given a subset $E \subseteq G$, the \emph{upper Banach density} of $E$, denoted by $d^*(E)$, is defined by
	\[ d^*(E):=\sup\{\bar{d}_{(\Phi_N)}(E) : (\Phi_N) \textrm{ is a F\o lner sequence in } G\},\]
	where, for a given F\o lner sequence $(\Phi_N)$, we put
	\[ \bar{d}_{(\Phi_N)}(E):=\limsup_{N \to \infty} \frac{|E \cap \Phi_N|}{|\Phi_N|}.\]
	\begin{corollary}\label{largesetspoly}
		Let $F$ be a countable field with characteristic zero. Let $E \subseteq F$ with $d^*(E)>0$, $k \in \mathbb{N}$, and $p_1,\dots,p_k \in F[x]$ be independent polynomials. Then, for each $\varepsilon>0$, the set $\{ n \in F : d^*(E \cap (E-p_1(n)) \cap \dots \cap (E-p_k(n)))>d^*(E)^{k+1}-\varepsilon\}$ is syndetic\footnote{We say that a set $S \subseteq F$ is \emph{syndetic} if $F=\bigcup_{n \in K} (S-n)$, where $K$ is a finite subset of $F$.}.
	\end{corollary}
	Because of the strength of Theorem~\ref{findep}, we also have a corollary in topological dynamics. We say that $(X,d,(T_n)_{n \in F})$ is a \emph{minimal topological dynamical system} if $(X,d)$ is a compact metric space, the maps $T_n: X \to X$, $n \in F,$ form an action of $F$ on $X$ by homeomorphisms, and for each $x \in X$ we have
	\[\overline{\{T_nx : n \in F\}}=X.\]
	\begin{corollary}\label{cor: findep}
		Let $F$ be a countable field with characteristic zero. Let $(X,d,\break (T_n)_{n \in F})$ be a minimal topological dynamical system, where $(X,d)$ is a compact metric space. Let $k \in \mathbb{N}$ and $p_1,\dots,p_k \in F[x]$ be a family of independent polynomials. Then there exists a dense $G_{\delta}$ set $X_0 \subseteq X$ such that for every $x \in X_0$ we have
		\begin{equation*} \overline{\{(T_{p_1(n)}x,\dots,T_{p_k(n)}x) : n \in F\}}=\underbrace{X \times \dots \times X}_{k\ \mathrm{times}}.
		\end{equation*}
	\end{corollary}
	\subsubsection{``Good'' rings}
	Our second application of Theorem~\ref{framainthm} will be to totally ergodic actions of a suitable class of rings which we call good rings and which we now define.
	\begin{definition}
		Let $R$ be a commutative ring. We say that $R$ is \emph{good} if it is a countable integral domain with characteristic zero, and every non-zero ideal has finite index in $R$.
	\end{definition}
	Defining this class of rings is meant to capture the properties of the ring $\mathbb{Z}$ that are desirable from the perspective of applications to ergodic theory. Good rings include countable fields of characteristic zero and also rings of integers of number fields. Our application of Theorem~\ref{framainthm} in the case of good rings is similar to that of countable fields of characteristic zero, but requires some slight modifications.
	
	Let us comment on the assumptions embedded in this definition. Countability is assumed for convenience. Working with rings of characteristic zero is meant to avoid technical trouble arising from torsion; the positive characteristic case is unknown to the authors. Similarly, in several places our argument requires non-zero ideals to have finite index; it is not known whether this assumption may be dispensed with. Lastly, it can be shown that (**) does not hold for countable rings with characteristic zero such that every non-zero ideal has finite index that are \emph{not} integral domains.
	We will prove the following theorem in Section~\ref{sec: 5}.
	\begin{theorem}\label{rindep}
		Let $R$ be a good ring. Let $(X,\mu,(T_r)_{r \in R})$ be a totally ergodic $R$-system. Let $(\Phi_N)$ be a F\o lner sequence in $(R,+)$, let $k \in \mathbb{N}$, let $p_1,\dots,p_k \in R[x]$ be independent polynomials, and let $f_1,\dots,f_k \in L^{\infty}(\mu)$ be functions. Then
		\[ \lim_{N \to \infty} \frac{1}{|\Phi_N|}\sum_{n\in \Phi_N} T_{p_1(n)}f_1\cdots T_{p_k(n)}f_k \ = \ \prod_{j=1}^k \int_X f_i \ d\mu\]
		in $L^2(\mu)$.
	\end{theorem}
	\begin{remark}
		As mentioned above, rings of integers of number fields are good rings. Therefore, Theorem~\ref{rindep} recovers \cite[Theorem~3.1]{ab} as a special case (see also item 2 in their abstract).
	\end{remark}
	
	Let $R$ be a countable ring with characteristic zero, and let $(T_n)_{n \in R}$ be an ergodic $R$-action on a probability space $(X,\mathcal{B},\mu)$. We will now sketch a proof of the fact that (**) can only hold if the action $(T_n)_{n \in R}$ is totally ergodic (it is obvious that ergodicity is necessary).
	
	Suppose to the contrary that (**) holds for the ergodic action $(T_n)_{n \in R}$, and assume that it is not totally ergodic. By assumption, we can find a finite index additive subgroup $J \subseteq R$, and a non-constant $f \in L^2(\mu)$ such that $T_nf=f$ for all $n \in J$. Since this $f$ has finite orbit in $L^2(\mu)$ under $(T_n)_{n \in R}$, we can write
	\[ f=\sum_{j \in \mathbb{N}} c_j f_j \quad \text{in } L^2(\mu) \]
	for some complex numbers $c_j$, not all zero, some functions $f_j: X \to \mathbb{S}^1$ with $T_nf_j=
	\chi_j(n)f_j$, where $\chi_j : R \to \mathbb{S}^1$ are distinct characters. It is easy to check that the $f_j$ are an orthogonal family, which implies that for each $j \in \mathbb{N}$, $\chi(n)=1$ for all $n \in J$. Moreover, since $J$ is of finite index, one can show that each $\chi_j$ only takes finitely many values, which are necessarily roots of unity (that may depend on $j$).
	
	Taking a suitable $f_j$, we see there exists a non-constant function $g: X \to \mathbb{S}^1$ such that $T_ng=
	\chi(n)g$, for some character $\chi : R \to \mathbb{S}^1$ taking values on finitely many roots of unity. Note that since $R$ has zero characteristic, we can find a natural number $a \in \mathbb{N}$ such that $\chi(an)=1$ for all $n \in R$. Finally, consider the polynomial $p(n)=an$ (which is non-trivial, since $R$ has characteristic zero). We have assumed that (**) holds; hence, from just the $k=1$ situation of (**), it follows that for any F\o lner sequence $(\Phi_N)$ in $R$ we have
	\[ \lim_{N \to \infty} \mathbb{E}_{n \in \Phi_N} T_{an}g=\lim_{N \to \infty} \mathbb{E}_{n \in \Phi_N} \chi(an)g=g.\]
	However, $\int_X g \ d\mu=0$, a contradiction.
	
	\subsection{Outline of the paper}
	This article is organized as follows. In Section~\ref{sec: 2} we introduce notation and recall some general facts in ergodic theory. In Section~\ref{sec: 3}, after some preparation, we prove Theorem~\ref{framainthm}. In Section~\ref{sec: 4}, we obtain Theorem~\ref{findep} via Theorem~\ref{framainthm}, which requires us to verify two conditions are met. We do this in Theorems~\ref{characterequidistribution}~and~\ref{BTZ factor is characteristic}, and the section ends with a few corollaries in combinatorics and topological dynamics. Lastly, Section~\ref{sec: 5} is devoted to the proof of Theorem~\ref{rindep} via Theorem~\ref{framainthm}.
	
	\section{Preliminaries}\label{sec: 2}
	\subsection{Notation}\label{sec: 2.1}
	We put $\mathbb{N}=\{1,2,3,\dots\}$. Let $G$ be a second-countable locally compact abelian group\footnote{It is classical that such groups possess F\o lner sequences.}. For any subset $A \subseteq G$ with $0<m(A)<\infty$ and any bounded function $a: G \to \mathbb{C}$ we write
	\[ \mathbb{E}_{n \in A} \ a(n) :=\frac{1}{m(A)}\int_A a(n) \ dm(n),\]
	where $m$ is a Haar measure on $G$.
	\begin{definition}
		Let $G$ be a second-countable locally compact abelian group with a Haar measure $m$.\footnote{Once the Haar measure $m$ has been fixed, we use $|A|$ and $m(A)$ interchangeably for measurable $A \subseteq G$.} We say that a sequence $(\Phi_N)$ of nonempty subsets of $G$ is a F\o lner sequence in $G$ if, for every $g \in G$, we have
		\[ \lim_{N \to \infty} \frac{m(\Phi_N \Delta (g+\Phi_N))}{m(\Phi_N)}=0.\]
		(Note that the definition is independent of the choice of Haar measure.)
	\end{definition}
	We use the notation $o_{M; N \to \infty}(1)$ to denote a quantity that, for fixed $M \in \mathbb{N}$, tends to $0$ as $N \to \infty$.
	
	Let $G$ be a second-countable locally compact abelian group. A \emph{measure-preserving system}, or simply \emph{a system}, is a quadruple $(X,\mathcal{B},\mu,(T_g)_{g \in G})$ for which $(X,\mathcal{B},\mu)$ is a probability space\footnote{All our spaces are standard, i.e. measurably isomorphic to an interval with Lebesgue measure together with (at most) countably many disjoint atoms.}, the map $T_g: X \to X$ is an invertible measure-preserving transformation for each $g \in G$, and $T_g\circ T_h=T_{gh}$ for all $g, h \in G$. We typically omit the $\sigma$-algebra $\mathcal{B}$ and write $(X,\mu,(T_g)_{g \in G})$. As usual, for $f \in L^{\infty}(\mu)$ and $g \in G$, we denote by $T_gf$ the function $f \circ T_g$.
	
	We say that $(X,\mu,(T_g)_{g \in G})$ is ergodic if the only functions $f \in L^1(\mu)$ satisfying $T_gf(x)=f(x)$ for $\mu$-a.e. $x \in X$ for all $g \in G$ are the constant functions.
	\subsection{Background in ergodic theory}
	We will need two variants of the van der Corput trick.
	\begin{lemma}\label{vdc}
		Let $\{x_g\}_{g\in G}$ be a bounded family of elements of a Hilbert space indexed by a second-countable locally compact abelian group $G$, and let $(\Phi_N)$ be a F\o lner sequence in $G$. Then, for all $M, N \in \mathbb{N}$,
		\[ \norm{\mathbb{E}_{n \in \Phi_N} x_n}^2 \leq \mathbb{E}_{h' \in \Phi_M}\mathbb{E}_{h \in \Phi_M-h'} \mathbb{E}_{n \in \Phi_N+h'} \langle x_{n+h}, x_n \rangle +o_{M;N \to \infty}(1).\]
	\end{lemma}
	\begin{proof}
		We slightly adapt the proof of \cite[Lemma~1.8]{abb}. Let $M \in \mathbb{N}$ and notice that
		\[ \norm{\mathbb{E}_{n \in \Phi_N} x_n}^2 \leq \norm{\mathbb{E}_{m \in \Phi_M}\mathbb{E}_{n \in \Phi_N} x_{n+m}}^2+\sup_{g \in G} \norm{x_g}^2 \mathbb{E}_{m \in \Phi_M} \frac{|\Phi_N \Delta (\Phi_N+m)|}{|\Phi_N|}.\]
		By Jensen's inequality, the right-hand side is bounded above by
		\begin{multline*} \mathbb{E}_{n \in \Phi_N}\norm{\mathbb{E}_{m \in \Phi_M} x_{n+m}}^2+o_{M;N \to \infty}(1) \\ =\mathbb{E}_{n \in \Phi_N}\mathbb{E}_{m_1 \in \Phi_M}\mathbb{E}_{m_2 \in \Phi_M} \langle x_{n+m_1}, x_{n+m_2} \rangle+o_{M;N \to \infty}(1).
		\end{multline*}
		To obtain the desired inequality, simply change variables and use Fubini's theorem (we can do this because the Haar measure is $\sigma$-finite) to exchange the \break expectations.
	\end{proof}
	The following variant comes from \cite[Lemma~4]{leibmanconvergence}. We note that the proof given in \cite{leibmanconvergence} applies to countable abelian groups $G$ without changes.
	\begin{lemma}\label{2D trick} Let $\{x_g\}_{g\in G}$ be a bounded family of elements of a Hilbert space indexed by a countable abelian group $G$, and let $(\Phi_N)$ be a F\o lner sequence in $G$.
		\begin{enumerate}
			\item For any finite set $S \subset G$,
			\begin{equation*}
				\limsup_{N \to \infty} \norm{\mathbb{E}_{g\in\Phi_N} x_g}^2 \leq \limsup_{N \to \infty} \mathbb{E}_{(h,h')\in S^2} \mathbb{E}_{g\in\Phi_N} \left\langle x_{g+h},x_{g+h'}\right\rangle \in \mathbb{R}.
			\end{equation*}
			\item There exists a F\o lner sequence $(\Theta_M)$ in $G^3$ such that
			\begin{equation*}
				\limsup_{N\to\infty} \norm{\mathbb{E}_{g\in\Phi_N} x_g}^2 \leq \limsup_{M\to\infty} \mathbb{E}_{(g,h,h')\in\Theta_M} \left\langle x_{g+h},x_{g+h'}\right\rangle \in \mathbb{R}.
			\end{equation*}
		\end{enumerate}
	\end{lemma}
	We will also make use of the following version of the mean ergodic theorem; see \cite[Theorem~4.15]{bergdiophantine} for a reference.
	\begin{theorem}[Mean ergodic theorem for countable abelian groups]\label{meanergthm} Let $G$ be a countable abelian group acting on a Hilbert space $\mathcal{H}$ by unitary transformations $U_g : \mathcal{H} \to \mathcal{H}$. Then, for any $f \in \mathcal{H}$ and any F\o lner sequence $(\Phi_N)$ of $G$, we have $\lim_{N\to\infty} \mathbb{E}_{g\in\Phi_N} U_g f = Pf$ in norm, where $P$ is the orthogonal projection onto the invariant subspace $\{ f' \in \mathcal{H} : U_g f' = f' \text{ for all } g \in G\}$.
	\end{theorem}
	
	\subsection{Joint ergodicity}
	
	\begin{definition}
		Let $G$ be a second-countable locally compact abelian group. Let $(\Phi_N)$ be a F\o lner sequence. We say that a collection of mappings $a_1,\dots,a_{\ell} : G \to G$ is
		\begin{itemize}
			\item[(i)] \emph{jointly ergodic (along $(\Phi_N)$) for the system} $(X,\mu,(T_g)_{g \in G})$ if, for all functions $f_1,\dots,f_{\ell} \in L^{\infty}(\mu)$, we have
			\begin{equation*}
				\lim_{N \to \infty} \mathbb{E}_{n \in \Phi_N} T_{a_1(n)}f_1\cdots T_{a_{\ell}(n)}f_{\ell} \ = \ \prod_{j=1}^{\ell} \int_X f_j \ d\mu,
			\end{equation*}
			where convergence takes place in $L^2(\mu)$.
			\item[(ii)] \emph{jointly ergodic (along $(\Phi_N)$)} if it is jointly ergodic for every ergodic system.
		\end{itemize}
	\end{definition}
	\begin{remark}
		When the group $G$ is countable, it is more convenient to think of the mappings $f: G \to G$, as well as the functions $h: G \to \mathbb{C}$, as sequences (in analogy with the terminology in \cite{frajointerg}).
	\end{remark}
	\begin{definition}
		Let $G$ be a second-countable locally compact abelian group. If $(X,\mu,(T_g)_{g \in G})$ is a system, we let
		\[ \mathrm{Spec}((T_g)_{g \in G}):=\{ \alpha\in\mathrm{Hom}(G,\mathbb{T}) : T_gf=e(\alpha(g))f \textrm{ for some non-zero } f \in L^2(\mu)\},\]
		where as usual $e(x) := \exp(2\pi i x)$ for real $x$, and we let
		\[ \mathcal{E}((T_g)_{g \in G}) := \{ f \in L^\infty(\mu) : |f| = 1 \text{ and } T_gf = e(\alpha(g))f  \text{ for some } \alpha \in \mathrm{Hom}(G,\mathbb{T})\}  \]
		denote the eigenfunctions of modulus 1.
	\end{definition}
	When convenient, we also write $T_gf=\chi(g)f$ for a character $\chi \in \hat{G}$ and say that $\chi$ is in the spectrum of $G$.
	
	The seminorms $\nnorm{\cdot}_s$, which appear in the next definition, are defined in Section~\ref{sec: 3.1}.
	\begin{definition}
		Let $G$ be a second-countable locally compact abelian group. We say that a collection of mappings $a_1,\dots,a_{\ell}: G \to G$ is:
		\begin{itemize}
			\item[(i)] \emph{good for seminorm estimates for the system} $(X,\mu,(T_g)_{g \in G})$ \emph{(along} $(\Phi_N)$\emph{)} if there exists $s \in \mathbb{N}$ such that whenever $f_1,\dots,f_{\ell} \in L^{\infty}(\mu)$ satisfies $\nnorm{f_l}_s=0$ for some $l \in \{1,\dots,\ell\}$ and $f_{l+1},\dots, f_\ell \in \mathcal{E}((T_g)_{g \in G})$, then
			\[ \lim_{N \to \infty} \mathbb{E}_{n \in \Phi_N} T_{a_1(n)}f_1\cdots T_{a_{\ell}(n)}f_{\ell}=0,\]
			where convergence takes place in $L^2(\mu)$. It is \emph{good for seminorm estimates} if it is good for seminorm estimates for every ergodic system.
			\item[(ii)] \emph{good for equidistribution for the system} $(X,\mu,(T_g)_{g \in G})$ \emph{(along} $(\Phi_N)$\emph{)} if for all $\alpha_1,\dots,\alpha_{\ell} \in \mathrm{Spec}((T_g)_{g \in G})$, not all of them trivial, we have
			\begin{equation}\label{equidistribution1}
				\lim_{N \to \infty} \mathbb{E}_{n \in \Phi_N} e(\alpha_1(a_1(n))+\dots+\alpha_{\ell}(a_{\ell}(n)))=0.
			\end{equation}
			It is \emph{good for equidistribution (along} $(\Phi_N)$\emph{)} if it is good for equidistribution (along $\Phi_N$) for every system, or equivalently, if \ref{equidistribution1} holds for all homomorphisms $\alpha_1,\dots,\alpha_{\ell}: G \to \mathbb{T}$, not all of them trivial.
		\end{itemize}
	\end{definition}
	\section{Sufficient conditions for joint ergodicity}\label{sec: 3}
	We first define some notation and prove some basic results about seminorms, following the work of \cite{frajointerg}.
	\subsection{Ergodic seminorms and related results} \label{sec: 3.1}
	\begin{definition}
		Let $G$ be a second-countable locally compact abelian group. Let $(X,\mu,(T_g)_{g \in G})$ be a system and $f \in L^{\infty}(\mu)$. Let $\underline{n}:=(n_1,\dots,n_s) \in G^s$, $\underline{n}':=(n_1',\dots,n_k') \in G^k$, $\varepsilon:=(\varepsilon_1,\dots,\varepsilon_s) \in \{0,1\}^s$, and $z \in \mathbb{C}$. We put
		\begin{itemize}
			\item[(i)] $\varepsilon \cdot \underline{n}:=\sum_{i=1}^s \varepsilon_i n_i$.
			\item[(ii)] $|\underline{n}|:=\sum_{i=1}^s n_i$ and $|\underline{\ep}|:=\sum_{i=1}^s \ep_i.$
			\item[(iii)] $\mathcal{C}^lz:=z$ if $l$ is even and $\mathcal{C}^lz:=\bar{z}$ if $l$ is odd.
			\item[(iv)] $\underline{n}^{\varepsilon}:=(n_1^{\varepsilon_1},\dots,n_s^{\varepsilon_s})$, where $n_j^0=n_j$, $n_j^1=n_j'$ for $j=1,\dots,s$. Note that this notation will always be used in a context where $\underline{n}'$ is understood.
			\item[(v)] $\Delta_{n_1}\cdots\Delta_{n_s}f:=\prod_{\varepsilon \in \{0,1\}^s} \mathcal{C}^{|\varepsilon|} T_{\varepsilon \cdot \underline{n}}f$. For example, $\Delta_n f = f \cdot T_n \overline{f}$.
			\item[(vi)] $\Delta_{\underline{n}}f:=\Delta_{n_1}\dots\Delta_{n_s}f$.
		\end{itemize}
	\end{definition}
	For a given system $(X,\mu,(T_g)_{g \in G})$, in complete analogy to the $\mathbb{Z}$ case, or indeed the $\mathbb{Z}^d$ case, the so-called Gowers--Host--Kra seminorms $\nnorm{\cdot}_s$ for $s \in \mathbb{N}$ can be defined. We are going to recall the inductive definition next, starting with $\nnorm{\cdot}_0$, which we define as follows for $f \in L^{\infty}(\mu)$:
	\[ \nnorm{f}_0:=\int_X f \ d\mu,\]
	and for $s \geq 0$ we put
	\begin{equation}\label{seminormdefinition}
		\nnorm{f}_{s+1}^{2^{s+1}}:=\lim_{N \to \infty} \mathbb{E}_{n \in \Phi_N} \nnorm{\Delta_nf}_s^{2^s},
	\end{equation}
	where $(\Phi_N)$ is any F\o lner sequence in $G$.	(Note that with our definition, $\nnorm{\cdot}_0$ is not a seminorm, but $\nnorm{\cdot}_s$ is for $s \geq 1$).
	The fact that the limit in \ref{seminormdefinition} exists can be shown using the mean ergodic theorem for $G$-actions repeatedly. Furthermore, for any $f \in L^{\infty}(\mu)$, the mean ergodic theorem and the Cauchy--Schwarz inequality imply that $\nnorm{f}_s \leq \nnorm{f}_{s+1}$ for all $s \geq 1$.

	If one repeatedly applies the inductive definition \ref{seminormdefinition}, one may express $\nnorm{f}_s^{2^s}$ as an iterated limit. By \cite[Lemmas~1.1, 1.2]{blcubicaverages} and \cite[Theorem~1.1]{zorinkranich}, such an iterated limit may be combined in a suitable way without changing the value of the limit. To give two examples, one may obtain
	\begin{equation*}
		\nnorm{f}_s^{2^s}=\lim_{N \to \infty} \mathbb{E}_{\underline{n} \in \Phi_N^s} \int_X \Delta_{\underline{n}} f \ d\mu
	\end{equation*}
	and, if one stops earlier,
	\begin{equation}\label{2seminormearlystop}
		\nnorm{f}_s^{2^s}=\lim_{N \to \infty} \mathbb{E}_{\underline{n} \in \Phi_N^{s-2}} \nnorm{\Delta_{\underline{n}}f}_2^4.
	\end{equation}
	An alternative way to define the seminorms is via the introduction of some measures on product spaces. Namely, given the system $(X,\mu,(T_n)_{n \in G})$, we consider the sequence of systems $(X^{[k]},\mu^{[k]},(T_n^{[k]})_{n \in G})$, for $k=0,1,2,\dots$. These are defined inductively as follows: $(X^{[0]},\mu^{[0]},(T_n^{[0]})_{n \in G}):=(X,\mu,(T_n)_{n \in G})$, and assuming that $(X^{[k]},\mu^{[k]},(T_n^{[k]})_{n \in G})$ has been defined, we let $\mathcal{I}_k$ be the $\sigma$-algebra of $(T_n^{[k]})_{n \in G}$-invariant subsets of $X^{[k]}$. Let $I_k$ be the factor associated to $\mathcal{I}_k$. Then, $(X^{[k+1]},\mu^{[k+1]})$ is defined as the relative product $(X^{[k]}, \mu^{[k]}) \times_{I_k} (X^{[k]},\mu^{[k]})$ with the natural action $(T_n^{[k+1]})_{n \in G}=(T_n^{[k]} \times T_n^{[k]})_{n \in G}$. In particular, for $F, H \in L^{\infty}(X^{[k]})$ we have
	\[ \int_{X^{[k+1]}} F \otimes H \ \ d\mu^{[k+1]}=\int_{X^{[k]}} \mathbb{E}[F|I_k]\cdot\mathbb{E}[H|I_k] \ d\mu^{[k]}.\]
	This allows for the following alternative definition of the ergodic seminorms. For $f \in L^{\infty}(\mu)$ and $s \geq 0$, we have\footnote{Given $f_{\varepsilon} \in L^{\infty}(\mu)$, we use the notation $\bigotimes_{\varepsilon \in \{0,1\}^s} f_{\varepsilon}:=\prod_{\varepsilon \in \{0,1\}^s} f_{\varepsilon}(x_{\varepsilon})$, in the product space $X^{2^s}$.}
	\[ \nnorm{f}_s:=\left(\int_{X^{[s]}} \bigotimes_{\varepsilon \in \{0,1\}^s} \mathcal{C}^{|\varepsilon|}f \ d\mu^{[s]}\right)^{\frac{1}{2^s}},\]
	where $\{0,1\}^{0}$ is taken to be the singleton $\{0\}$.

In the sequel, we will need the following lemma. We denote by $\mathcal{K}((T_g)_{g \in G})$ the Kronecker factor of $(X,\mu,(T_g)_{g \in G})$.
	\begin{lemma}\label{2seminormkronecker}
		Let $G$ be a second-countable locally compact abelian group. Let $(X,\mu,(T_g)_{g \in G})$ be an ergodic system and $f \in L^{\infty}(\mu)$. Let $\tilde{f}:=\mathbb{E}[f|\mathcal{K}((T_g)_{g \in G})]$. Then
		\begin{equation*}
			\nnorm{f}_2=\nnorm{\tilde{f}}_2.
		\end{equation*}
	\end{lemma}
	\begin{proof}
		Note that by the definition of the seminorms, we have
		\begin{multline}
			\label{deleeuw1}
			\nnorm{f}_2^4=\lim_{N_1 \to \infty} \mathbb{E}_{n_1 \in \Phi_{N_1}} \lim_{N_2 \to \infty} \mathbb{E}_{n_2 \in \Phi_{N_2}} \int_X f\cdot T_{n_1}\bar{f}\cdot T_{n_2}\bar{f}\cdot T_{n_1+n_2}f \ d\mu= \\
			\lim_{N_1 \to \infty} \mathbb{E}_{n_1 \in \Phi_{N_1}} \int_{X \times X} (f \otimes \bar{f})\cdot (T \times T)_{n_1}(\bar{f} \otimes f) \ d\mu \ d\mu.
		\end{multline}
		Using the Jacobs--de Leeuw--Glicksberg decomposition\footnote{There is a hint of this decomposition (when $G = \mathbb{Z}$) in the work of Koopman and von Neumann \cite{koopmanvonneumann}. See \cite[Section~16.3]{efhn} for a more modern treatment.}, we see by ergodicity that the last term in the equality \ref{deleeuw1} is $0$ in case $f \in L^2(\mu)_{\textrm{wm}}$ (as it implies that $f \otimes \bar{f}$ is in the weak mixing component of the product space). Thus, $\nnorm{f}_2^4=\nnorm{Pf}_2^4$, where $P$ is the projection onto the compact component of $L^2(\mu)$. It can be shown that $Pf=\tilde{f}$ $\mu$-a.e., so taking 4th roots completes the proof.
	\end{proof}
	This lemma allows us to generalize \cite[Proposition~3.1]{frajointerg} to derive the following proposition.
	\begin{proposition}\label{correlates1}
		Let $G$ be a second-countable locally compact abelian group. Let $(X,\mu,(T_g)_{g \in G})$ be an ergodic system and $f \in L^{\infty}(\mu)$ be a function with $\norm{f}_{L^{\infty}(\mu)} \leq 1$. Then
		\[ \nnorm{f}_2^4 \ \leq \  \sup_{\chi \in \mathcal{E}((T_g)_{g \in G})} \mathrm{Re} \left( \int_X f \cdot \chi \ d\mu\right).\]
	\end{proposition}
	\begin{proof}
		By Lemma \ref{2seminormkronecker} we have
		\[ \nnorm{f}_2=\nnorm{\tilde{f}}_2.\]
		Since $(T_g)_{g \in G}$ is ergodic, $\mathcal{K}((T_g)_{g \in G})$ has an orthonormal basis of eigenfunctions of modulus one. Furthermore, since $(X,\mathcal{B},\mu)$ is standard (defined in footnote 8), such a base is countable; suppose it is $(\chi_j)_{j \in \mathbb{N}}$. The rest of the proof is as in \cite[Proposition~3.2]{frajointerg} verbatim, so we do not reproduce it here.
	\end{proof}
	
	We may use the previous proposition to deduce the following one.
	\begin{proposition}
		Let $G$ be a second-countable locally compact abelian group. Let $(\Phi_N)$ be a F\o lner sequence. Let $(X,\mu,(T_g)_{g \in G})$ be an ergodic system and $f \in L^{\infty}(\mu)$ be such that $\nnorm{f}_{s+2}>0$ for some $s \geq 0$. Then there exist $\chi_{\underline{n}} \in \mathcal{E}((T_g)_{g \in G})$, $\underline{n} \in G^s$, such that
		\[ \liminf_{N \to \infty} \mathbb{E}_{\underline{n} \in \Phi_N^s} \mathrm{Re}\left( \int_X \Delta_{\underline{n}}f \cdot \chi_{\underline{n}} \ d\mu\right)>0.\]
	\end{proposition}
	\begin{proof}
		Without loss of generality we will assume that $\norm{f}_{L^{\infty}(\mu)} \leq 1$.	We begin with an application of \ref{2seminormearlystop} to see that the initial hypothesis is equivalent to:
		\[ \lim_{N \to \infty} \mathbb{E}_{\underline{n} \in \Phi_N^s} \nnorm{\Delta_{\underline{n}}f}_2^4>0.\]
		A simple application of Proposition~\ref{correlates1} shows that for some $a>0$ we have
		\[ \liminf_{N \to \infty} \mathbb{E}_{\underline{n} \in \Phi_N^s} \sup_{\chi \in \mathcal{E}((T_g)_{g \in G})} \mathrm{Re} \left( \int_X \Delta_{\underline{n}}f \cdot \chi \ d\mu\right)\geq a. \]
		Now, for each $\underline{n} \in G^s$, choose $\chi_{\underline{n}} \in \mathcal{E}((T_g)_{g \in G})$ such that
		\[ \mathrm{Re} \left( \int_X \Delta_{\underline{n}}f \cdot \chi_{\underline{n}} \ d\mu\right) \geq
		\sup_{\chi \in \mathcal{E}((T_g)_{g \in G})} \mathrm{Re} \left( \int_X \Delta_{\underline{n}}f \cdot \chi \ d\mu\right)-\frac{a}{10},\]
		which completes the proof.
	\end{proof}
	Up to straightforward changes in notation we also have the following lemma (cf. \cite[Lemma~3.3]{frajointerg}):
	\begin{lemma}
		Let $G$ be a second-countable locally compact abelian group. Let $(\Phi_N)$ be a F\o lner sequence. Let $(X,\mu,(T_g)_{g \in G})$ be a system, and for any $s \in \mathbb{N}$ let $f_{\varepsilon} \in L^{\infty}(\mu)$ be bounded by $1$ for $\varepsilon \in \{0,1\}^s$, and $g_{\underline{n}} \in L^{\infty}(\mu)$ for $\underline{n} \in G^s$. Let $\underline{1}:=(1,\dots,1)$. Then for every $N \in \mathbb{N}$ we have
		\begin{multline*} \left| \mathbb{E}_{\underline{n} \in \Phi_N^s} \int_X \prod_{\varepsilon \in \{0,1\}^s} T_{\varepsilon \cdot \underline{n}} f_{\varepsilon} \cdot g_{\underline{n}} \ d\mu\right|^{2^s} \\ \leq \mathbb{E}_{\underline{n},\underline{n'} \in \Phi_N^s} \int_X \Delta_{\underline{n}-\underline{n}'} f_{\underline{1}} \cdot T_{-|\underline{n}|} \left( \prod_{\varepsilon \in \{0,1\}^s} \mathcal{C}^{|\varepsilon|} g_{\underline{n}^{\varepsilon}}\right) \ d\mu.\end{multline*}
	\end{lemma}
	The analogue of \cite[Lemma~3.4]{frajointerg} also holds in our context, but the proof has to be adapted at a certain stage:
	\begin{lemma}\label{stepreduction}
		Let $G$ be a second-countable locally compact abelian group. Let $(X,\mu,(T_g)_{g \in G})$ be an ergodic system and $f \in L^{\infty}(\mu)$ be such that $\nnorm{f}_s=0$ for some $s \in \mathbb{N}$. Let $(\Phi_N), (\Psi_N)$ be F\o lner sequences in $G$. For $j=1,\dots,s$, $N \in \mathbb{N}$, let $b_{j,N}(n_1,\dots,n_s)\in \ell^{\infty}(G^s)$ be mappings that do not depend on the variable $n_j$ and are bounded by $1$, and let
		\[ c_N(\underline{n}):=\prod_{j=1}^s b_{j,N}(\underline{n}), \ \underline{n} \in \Phi_N^s, N \in \mathbb{N}.\]
		Then,
		\[ \lim_{D \to \infty}\lim_{N \to \infty}\norm{\mathbb{E}_{\underline{n} \in \Phi_N^{s-1}}\mathbb{E}_{d \in \Psi_D} c_N(\underline{n},d) \cdot \Delta_{\underline{n},d}f}_{L^2(\mu)}=0. \]
	\end{lemma}
	\begin{proof}
		Using the Cauchy--Schwarz inequality, together with the fact that $b_{s,N}$ does not depend on $n_s$ and is bounded by $1$ and the fact that
		\[\Delta_{\underline{n},d}f=T_{d}(\Delta_{\underline{n}}\bar{f})\cdot (\Delta_{\underline{n}}f), \quad \underline{n} \in G^{s-1}, d \in G, \]
		it suffices to show that
		\begin{equation}\label{eqn: stupid} \lim_{D \to \infty}\lim_{N \to \infty} \mathbb{E}_{\underline{n} \in \Phi_N^{s-1}} \norm{\mathbb{E}_{n_s' \in \Psi_D} \prod_{j=1}^{s-1}b_{j,N}(\underline{n},n_s')\cdot T_{n_s'}(\Delta_{\underline{n}}\bar{f})}^2_{L^2(\mu)}=0,
		\end{equation}
		relabeling $d$ as $n_s'$.
		
		By Lemma \ref{vdc} for the average over $n_s'$ and composing with $T_{-n_s'}$, for any $M \in \mathbb{N}$, the left-hand side of \ref{eqn: stupid} is bounded above by
		\begin{multline*} \lim_{D \to \infty}\lim_{N \to \infty} \mathbb{E}_{\underline{n} \in \Phi_N^{s-1}} \mathbb{E}_{h' \in \Phi_M}\mathbb{E}_{n_s \in \Phi_M-h'} \mathbb{E}_{n_s' \in \Psi_D+h'} \\ \prod_{j=1}^{s-1} b_{j,N}(\underline{n},n_s'+n_s)\cdot\bar{b}_{j,N}(\underline{n},n_s') \int_X T_{n_s}(\Delta_{\underline{n}}\bar{f})\cdot \Delta_{\underline{n}} f \ d\mu, \end{multline*}
		which is bounded above by
		\begin{multline*} \sup_{\substack{(H_N),(\Theta_D) \\ \text{F\o lner sequences}}} \Bigg|\lim_{D \to \infty}\lim_{N \to \infty} \mathbb{E}_{\underline{n} \in \Phi_N^{s-1}}\mathbb{E}_{n_s \in H_M} \mathbb{E}_{n_s' \in \Theta_D} \\ \prod_{j=1}^{s-1} b_{j,N}(\underline{n},n_s'+n_s)\cdot\bar{b}_{j,N}(\underline{n},n_s') \int_X T_{n_s}(\Delta_{\underline{n}}\bar{f})\cdot \Delta_{\underline{n}} f \ d\mu\Bigg|. \end{multline*}
		Since $M$ is arbitrary, it is thus enough to show that for all F\o lner sequences $(H_N), (\Theta_N)$ we have
		\begin{multline*} \lim_{M \to \infty}\lim_{D \to \infty}\lim_{N \to \infty} \mathbb{E}_{\underline{n} \in \Phi_N^{s-1}}\mathbb{E}_{n_s \in H_M} \mathbb{E}_{n_s' \in \Theta_D} \\ \prod_{j=1}^{s-1} b_{j,N}(\underline{n},n_s'+n_s)\cdot\bar{b}_{j,N}(\underline{n},n_s') \int_X T_{n_s}(\Delta_{\underline{n}}\bar{f})\cdot \Delta_{\underline{n}} f \ d\mu \ = \ 0,\end{multline*}
		so it is enough to show that
		\begin{multline*} \lim_{M \to \infty}\lim_{D \to \infty}\lim_{N \to \infty}
			\mathbb{E}_{n_s' \in \Theta_D} \Bigg|\mathbb{E}_{\underline{n} \in \Phi_N^{s-1}}\mathbb{E}_{n_s \in H_M} \\  \prod_{j=1}^{s-1} b_{j,N}(\underline{n},n_s'+n_s)\cdot\bar{b}_{j,N}(\underline{n},n_s') \int_X T_{n_s}(\Delta_{\underline{n}}\bar{f})\cdot \Delta_{\underline{n}} f \ d\mu\Bigg| \ = \ 0.\end{multline*}
		Thus, as in \cite[Lemma 3.4]{frajointerg}, it is enough to show that for any choice of $n_{s,D}' \in \Theta_D, D  \in \mathbb{N}$ we have
		\begin{multline*} \lim_{M \to \infty} \lim_{N \to \infty} \mathbb{E}_{\underline{n} \in \Phi_N^{s-1}}\mathbb{E}_{n_s \in H_M} \\ \prod_{j=1}^{s-1} b_{j,N}(\underline{n},n_{s,D}'+n_s)\cdot\bar{b}_{j,N}(\underline{n},n_{s,D}') \int_X T_{n_s}(\Delta_{\underline{n}}\bar{f})\cdot \Delta_{\underline{n}} f \ d\mu \ = \ 0, \end{multline*}
		and by the Cauchy--Schwarz inequality it is actually enough to show that
		\[ \lim_{M \to \infty}\lim_{N \to \infty} \norm{\mathbb{E}_{\underline{n} \in \Phi_N^{s-1}}\mathbb{E}_{n_s \in H_M}\prod_{j=1}^{s-1} b_{j,N}'(\underline{n},n_s)\cdot \Delta_{\underline{n},n_s}f}_{L^2(\mu)}=0,\]
		where $b_{j,N}'(\underline{n},n_s):=b_{j,N}(\underline{n},n_{s,D}'+n_s)\cdot\bar{b}_{j,N}(\underline{n},n_{s,D}').$ Note that we now have a product of $s-1$ mappings $b_{j,N}'$ that do not depend on the variable $n_j$. If we keep iterating this process, we will find that it is enough to show
		\[ \lim_{M \to \infty}\lim_{N \to \infty} \mathbb{E}_{\underline{n} \in \Phi_N^{s-1}}\mathbb{E}_{n_s \in H_M} \int_X \Delta_{\underline{n},n_s}f \ d\mu=0\]
		for any choice of F\o lner sequences $(\Phi_{N}),(H_{N})$. This now follows from the assumption that $\nnorm{f}_s=0$, as the choice of F\o lner sequence does not affect the definition.
	\end{proof}
	\subsection{Proof of Theorem~\ref{framainthm}}
	We follow the approach in \cite{frajointerg} and obtain the following result. It is restated here in a more complicated way for technical reasons.
	\begin{thmrep}[\ref{framainthm}]
		Let $G$ be a second-countable locally compact abelian group. Let $(\Phi_N)$ be a F\o lner sequence. Let $(X,\mu,(T_g)_{g \in G})$ be an ergodic system and $a_1,\dots,a_{\ell} : G \to G$ be mappings that are good for seminorm estimates and equidistribution for $(X,\mu,(T_g)_{g \in G})$. Then, for every $l \in \{0,1,\dots,\ell\}$ we have:
		\begin{itemize}
			\item[$(P_l)$] If $f_1,\dots,f_{\ell} \in L^{\infty}(\mu)$ are such that $f_j \in \mathcal{E}((T_g)_{g \in G})$ for $j=l+1,\dots,\ell$, then
			\begin{equation}\label{mainthmeq}
				\lim_{N \to \infty} \mathbb{E}_{n \in \Phi_N} T_{a_1(n)}f_1\cdots T_{a_{\ell}(n)} f_{\ell} \ = \ \int_X f_1 \ d\mu \cdots \int_X f_{\ell} \ d\mu \end{equation}
			in $L^2(\mu)$.
		\end{itemize}
	\end{thmrep}
	In order to prove Theorem~\ref{framainthm} we need some auxiliary results. The first one is an easy proof by contradiction as in \cite[Proposition~4.2]{frajointerg}, so we merely give the statement.
	\begin{proposition}\label{flprop}
		Let $G$ be a second-countable locally compact abelian group. Let $a_1,\dots,a_{\ell}: G \to G$ be mappings, $(\Phi_N)$ be a F\o lner sequence, $(X,\mu, (T_g)_{g \in G})$ be a system, and $f_1,\dots,f_{\ell} \in L^{\infty}(\mu)$ be such that
		\[ \limsup_{N\to\infty} \norm{\mathbb{E}_{n \in \Phi_N} T_{a_1(n)}f_1\cdots T_{a_{\ell}(n)}f_{\ell}}_{L^2(\mu)} \ > \ 0. \]
		Let $l \in \{1,\dots,\ell\}$. Then we can find a F\o lner subsequence $(\Phi_{N_k})$ and $g_k \in L^{\infty}(\mu)$ with $\norm{g_k}_{L^{\infty}(\mu)} \leq 1$, $k \in \mathbb{N}$, such that for
		\begin{equation}\label{fldef}
			\tilde{f}_{l}:=\lim_{k \to \infty} \mathbb{E}_{n \in \Phi_{N_k}} T_{-a_{l}(n)}g_k \cdot \prod_{j=1, j \neq l}^{\ell} T_{a_j(n)-a_l(n)}\bar{f_j},
		\end{equation}
		where the limit is a weak limit (so $\tilde{f}_l \in L^{\infty}(\mu)$), we have
		\begin{equation*}
			\limsup_{N\to\infty} \norm{\mathbb{E}_{n \in \Phi_N} T_{a_l(n)}\tilde{f}_l \cdot \prod_{j=1, j \neq l}^{\ell} T_{a_j(n)}f_j}_{L^2(\mu)} \ > \  0.
		\end{equation*}
		Furthermore, if $\ell = l = 1$ and $\int f_1 \ d\mu = 0$, then we can choose $\tilde{f}_1$ of the form \ref{fldef} and such that $\int \tilde{f}_1 \ d\mu = 0$. 
	\end{proposition}
	For the next result, the proof given in \cite[Proposition~4.3]{frajointerg} works for our setting almost verbatim, so we do not reproduce it.
	\begin{proposition}\label{fnkprop}
		Let $G$ be a second-countable locally compact abelian group. Let $(X,\mu,(T_g)_{g \in G})$ be an ergodic system, $f_{n,k} \in L^{\infty}(\mu)$, $k \in \mathbb{N}, n \in G$, be bounded by $1$, and $f \in L^{\infty}(\mu)$ be defined by
		\[ f:=\lim_{k \to \infty} \mathbb{E}_{n \in \Phi_{N_k}} f_{n,k}\]
		for some F\o lner sequence $(\Phi_{N_k})$, where the average is assumed to converge weakly. If $\nnorm{f}_{s+2}>0$ for some $s \geq 0$, then there exist $a>0$, a subset $\Lambda$ of $G^s$ with $\underline{d}_{(\Phi_N^s)}(\Lambda)>0$, and $\chi_{\underline{n}} \in \mathcal{E}((T_g)_{g \in G})$, $\underline{n} \in G^s$, \footnote{We choose $\chi_{\underline{n}}$, $\underline{n} \in G^s$, arbitrarily if $\underline{n} \notin \Lambda$.} such that
		\[ \mathrm{Re} \left( \int_X \Delta_{\underline{n}}f \cdot \chi_{\underline{n}}\right)>a, \quad \underline{n} \in \Lambda,\]
		and
		\[ \liminf_{N \to \infty} \mathbb{E}_{n, n' \in \Phi_N^s} \limsup_{k \to \infty} \mathbb{E}_{n \in \Phi_{N_k}} \mathrm{Re} \left(\int_X \Delta_{\underline{n}-\underline{n}'}f_{n,k} \cdot \chi_{\underline{n},\underline{n}'}\cdot 1_{\Lambda'}(\underline{n},\underline{n}')\right)>0,\]
		where
		\[ \chi_{\underline{n},\underline{n}'}:= T_{-|\underline{n}|} \left( \prod_{\varepsilon \in \{0,1\}^s} \mathcal{C}^{|\varepsilon|}\chi_{\underline{n}^{\varepsilon}}\right),\]
		and
		\[ \Lambda':=\{(\underline{n},\underline{n}') \in G^{2s} : \underline{n}^{\varepsilon} \in \Lambda \textrm{ for all } \varepsilon \in \{0,1\}^s\}.\]
	\end{proposition}
	We now move to the most important intermediate result in this section:
	\begin{proposition}\label{mainprop}
		Let $G$ be a second-countable locally compact abelian group. Let $(\Phi_N)$ be a F\o lner sequence. Let $(X,\mu,(T_g)_{g \in G})$ be an ergodic system and $a_1,\dots,a_{\ell} : G \to G$ be good for equidistribution for the system $(X,\mu,(T_g)_{g \in G})$. Suppose that property $(P_{l-1})$ of Theorem \ref{framainthm} holds for some $l \in \{1,\dots,\ell\}$. Let $f_1,\ldots, f_{\ell}$ be as in Proposition~\ref{flprop}, and let $(\Phi_{N_k})$, $g_k$, $k\in \mathbb{N}$, and $\tilde{f}_{l}$ be as in the conclusion of that proposition. Further assume that $f_{l+1},\dots,f_{\ell} \in \mathcal{E}((T_g)_{g\in G})$. Suppose that $\nnorm{\tilde{f}_l}_{s'+2}>0$ for some $s' \geq 0$. Then $\int_X \tilde{f}_l \ d\mu \neq 0$.
	\end{proposition}
	\begin{proof}
		Suppose without loss of generality that the functions $f_1,\dots,f_{\ell}$ are bounded by $1$. We will show that if $\nnorm{\tilde{f}_l}_{s+2}>0$ for some $s \geq 0$, then $\nnorm{\tilde{f}_l}_{s+1}>0$. Applying this $s'+1$ times, it follows that $\nnorm{\tilde{f}_l}_1>0$, so $\int_X \tilde{f}_l \ d\mu \neq 0$ by ergodicity. Since $\nnorm{\tilde{f}_l}_{s+2}>0$, we can use Proposition \ref{fnkprop} for $\tilde{f}_l$ with
		\[ f_{n,k}:=T_{-a_l(n)}g_k \cdot \prod_{j \in \{1,\dots,\ell\}, j \neq l} T_{a_j(n)-a_l(n)} \bar{f}_j, \quad k \in \mathbb{N}, n \in G.\]
		Thus, there exist $a>0$, a subset $\Lambda$ of $G^s$ with positive lower density (along $(\Phi_N^s)$), $\chi_{\underline{n}} \in \mathcal{E}((T_g)_{g \in G})$, $\underline{n} \in G^s$, such that
		\begin{equation}\label{firstineqmainprop}
			\mathrm{Re} \left( \int_X \Delta_{\underline{n}} \tilde{f}_l \cdot \chi_{\underline{n}}\right)>a, \quad \underline{n} \in \Lambda,
		\end{equation}
		and
		\begin{multline}\label{mainpropfirsteq}
			\liminf_{N \to \infty} \mathbb{E}_{\underline{n}, \underline{n}' \in \Phi_N^s} \limsup_{k \to \infty}
			\mathbb{E}_{n \in \Phi_{N_k}} \mathrm{Re} \Bigg( \int_X T_{-a_l(n)}(\Delta_{\underline{n}-\underline{n}'}g_k) \\ \prod_{j \in \{1,\dots,\ell\}, j \neq l} T_{a_j(n)-a_l(n)}(\Delta_{\underline{n}-\underline{n}'}\bar{f}_j) \cdot \chi_{\underline{n},\underline{n}'} \cdot 1_{\Lambda'}(\underline{n},\underline{n}') \ d\mu\Bigg) \ > \ 0,
		\end{multline}
		where
		\[ \chi_{\underline{n},\underline{n}'}:=T_{-|\underline{n}|} \left( \prod_{\varepsilon \in \{0,1\}^s}\mathcal{C}^{|\varepsilon|} \chi_{\underline{n}^{\varepsilon}} \right), \quad \underline{n}, \underline{n}' \in G^s,\]
		and
		\[ \Lambda':=\{(\underline{n},\underline{n}') \in G^{2s} : \underline{n}^{\varepsilon} \in \Lambda \textrm{ for all } \varepsilon \in \{0,1\}^s \}.\]
		We start by analyzing \ref{mainpropfirsteq}. After composing with $T_{a_l(n)}$ we get
		\begin{multline*}
			\liminf_{N \to \infty} \mathbb{E}_{\underline{n}, \underline{n}' \in \Phi_N^s} \limsup_{k \to \infty}
			\mathbb{E}_{n \in \Phi_{N_k}} \mathrm{Re} \Bigg( \int_X (\Delta_{\underline{n}-\underline{n}'}g_k) \\ \prod_{j \in \{1,\dots,\ell\}, j \neq l} T_{a_j(n)}(\Delta_{\underline{n}-\underline{n}'}\bar{f}_j) \cdot T_{a_l(n)}\chi_{\underline{n},\underline{n}'} \cdot 1_{\Lambda'}(\underline{n},\underline{n}') \ d\mu\Bigg) \ > \ 0.
		\end{multline*}
		Let
		\[ g_{j,\underline{n},\underline{n}'}:=\Delta_{\underline{n}-\underline{n}'}\bar{f}_j, \quad j \in \{1,\dots,\ell\}, \quad j \neq l, \quad \underline{n}, \underline{n}' \in G^s,\]
		and
		\[ g_{l,\underline{n},\underline{n}'}:=\chi_{\underline{n},\underline{n}'}, \quad \underline{n},\underline{n}' \in G^s.\]
		Using Cauchy--Schwarz we deduce that
		\[ \liminf_{N \to \infty} \mathbb{E}_{\underline{n}, \underline{n}' \in \Phi_N^s} 1_{\Lambda'}(\underline{n},\underline{n}') \limsup_{k \to \infty} \norm{\mathbb{E}_{n \in \Phi_{N_k}} \prod_{j=1}^{\ell} T_{a_j(n)} g_{j,\underline{n},\underline{n}'}}_{L^2(\mu)}>0.\]
		Since
		\[ 1_{\Lambda'}(\underline{n},\underline{n}')=\prod_{\varepsilon \in \{0,1\}^s} 1_{\Lambda}(\underline{n}^{\varepsilon}) \leq 1_{\Lambda}(\underline{n})\]
		and the set $\Lambda$ has positive lower density, we get
		\begin{equation}\label{eqn: some garbage}
			\liminf_{N \to \infty} \mathbb{E}_{\underline{n}' \in \Phi_N^s} \mathbb{E}_{\underline{n} \in \Lambda \cap \Phi_N^s} \limsup_{k \to \infty} \norm{\mathbb{E}_{n \in \Phi_{N_k}}\prod_{j=1}^{\ell} T_{a_j(n)} g_{j,\underline{n},\underline{n}'}}_{L^2(\mu)}>0.
		\end{equation}
		By assumption, $f_j \in \mathcal{E}((T_g)_{g \in G})$ for $j=l+1,\dots,\ell$, so it follows that $g_{j,\underline{n},\underline{n}'} \in \mathcal{E}((T_g)_{g \in G})$ for $j=l+1,\dots,\ell$, $\underline{n},\underline{n}' \in G^s$. Moreover, since $\chi_{\underline{n}} \in \mathcal{E}((T_g)_{g \in G})$ for all $\underline{n} \in G^s$, we get that $g_{l,\underline{n},\underline{n}'} \in \mathcal{E}((T_g)_{g \in G})$ for all $\underline{n},\underline{n}' \in G^s$. By assumption, property $(P_{l-1})$ of Theorem \ref{framainthm} holds, so \ref{eqn: some garbage} remains unchanged if for $j=1,\dots,l-1$ we replace the functions $g_{j,\underline{n},\underline{n}'}, \underline{n},\underline{n}' \in G^s,$ by their integrals.
		Now, for some group homomorphisms $\alpha_{\underline{n}}: G \to \mathbb{T}$ we have
		\begin{equation*}
			T_{a_l(n)}\chi_{\underline{n}}=e(\alpha_{\underline{n}}(a_l(n)))\chi_{\underline{n}}, \quad n \in G, \underline{n} \in G^s,
		\end{equation*}
		where $\alpha_{\underline{n}}$ is in the spectrum of $(T_g)_{g \in G}$, $\underline{n} \in G^s$. Thus,
		\[ T_{a_l(n)}g_{l,\underline{n},\underline{n}'}=e(\beta_{\underline{n},\underline{n}'}(a_l(n)))g_{l,\underline{n},\underline{n}'},\]
		where
		\[ \beta_{\underline{n},\underline{n}'}:=\sum_{\varepsilon \in \{0,1\}^s} (-1)^{|\varepsilon|} \alpha_{\underline{n}^{\varepsilon}}.\]
		For $j=l+1,\dots,\ell, \underline{n},\underline{n}' \in G^s$, the eigenvalues of the eigenfunctions $g_{j,\underline{n},\underline{n}'}$ are of the form $e(\alpha_{j,\underline{n},\underline{n}'})$, where $\alpha_{j,\underline{n},\underline{n}'} : G \to \mathbb{T}$ are in the spectrum of the $G$-action. Thus,
		\begin{equation}\label{characterspositive}
\begin{split}
			&\liminf_{N \to \infty} \mathbb{E}_{\underline{n}' \in \Phi_N^s} \mathbb{E}_{\underline{n} \in \Lambda \cap \Phi_N^s} \limsup_{k \to \infty} \left| \mathbb{E}_{n \in \Phi_{N_k}} e\left(\beta_{\underline{n},\underline{n}'}(a_l(n))+\sum_{j=l+1}^{\ell} \alpha_{j,\underline{n},\underline{n}'}(a_j(n))\right)\right|\\&>0.
\end{split}
		\end{equation}
		From \ref{characterspositive} we can find elements $\underline{n}'_N \in \Phi_N^s$, and subsets $\Lambda_N$ of $\Lambda \cap \Phi_N^s$ with
		\begin{equation}\label{positivedensitymainprop}
			\liminf_{N \to \infty} \frac{|\Lambda_N|}{|\Phi_N|^s}>0, \quad N \in \mathbb{N},
		\end{equation}
		and such that
		\[ \limsup_{k \to \infty} \left| \mathbb{E}_{n \in \Phi_{N_k}} e(\beta_{\underline{n},\underline{n}'_N}(a_l(n))+\sum_{j=l+1}^{\ell} \alpha_{j,\underline{n},\underline{n}'_N}(a_j(n)))\right|>0, \quad \underline{n} \in \Lambda_N.\]
		Since the mappings $a_1,\dots,a_{\ell}$ are good for equidistribution for $(X,\mu,(T_g)_{g \in G})$ and, for every $j=l+1,\dots,\ell$, $\underline{n} \in \Lambda_N$, and $N \in \mathbb{N}$, the homomorphisms $\beta_{\underline{n},\underline{n}'_N}$, $\alpha_{j,\underline{n},\underline{n}'_N}$ are in the spectrum of $(T_g)_{g \in G}$, we deduce that
		\[ \beta_{\underline{n},\underline{n}'_N}=0, \quad \underline{n} \in \Lambda_N, N \in \mathbb{N},\]
		that is, it is the trivial homomorphism. Now, if $\underline{n}_N^{\varepsilon}:=(n_1^{\varepsilon_1},\dots,n_s^{\varepsilon_s})$, where $n_j^0:=n_j$ and $n_j^1:=n_{j,N}'$ for $j=1,\dots,s$, on recalling the definition of $\beta_{\underline{n},\underline{n}'}$, we observe
		\begin{equation}\label{characterequalitylast}
			\alpha_{\underline{n}}=-\sum_{\varepsilon \in \{0,1\}^s \setminus \underline{0}} (-1)^{|\varepsilon|}\alpha_{\underline{n}_N^{\varepsilon}}, \quad \underline{n} \in \Lambda_N, N \in \mathbb{N}.
		\end{equation}
		This implies that $\alpha_{\underline{n}}$ can be expressed as a sum of mappings, each of which depends only on $s-1$ variables from $n_1,\dots,n_s$. Since \ref{firstineqmainprop} holds for all $\underline{n} \in \Lambda_N \subseteq \Lambda$, $N \in \mathbb{N}$, and \ref{positivedensitymainprop} holds, we deduce that
		\[ \liminf_{N \to \infty} \mathbb{E}_{\underline{n} \in \Lambda_N} \mathrm{Re}\left( \int_X \Delta_{\underline{n}} \tilde{f}_l \cdot \chi_{\underline{n}} \ d\mu\right)>0.\]
		Composing with $T_{n_s'}$ and averaging $n_s'$ over $\Phi_N$, together with the fact that
		\[ T_{n_s'}\chi_{\underline{n}}=e(\alpha_{\underline{n}}(n_s')) \cdot \chi_{\underline{n}}, \quad \underline{n} \in G^s, n_s' \in G,\]
		and then using the Cauchy--Schwarz inequality, we deduce that
		\[ \liminf_{N \to \infty} \mathbb{E}_{\underline{n} \in \Lambda_N} \norm{\mathbb{E}_{n_s' \in \Phi_N} T_{n_s'}(\Delta_{\underline{n}}\tilde{f}_l) \cdot e(\alpha_{\underline{n}}(n_s'))}_{L^2(\mu)}>0.\]
		Using Lemma \ref{vdc} for the average over $n_s'$ and composing with $T_{-n_s'}$ on the integrals that arise we deduce that for each $M \in \mathbb{N}$,
		\[ \limsup_{N \to \infty} \mathrm{Re} \left( \mathbb{E}_{\underline{n} \in \Phi_N^{s}} \mathbb{E}_{h' \in \Phi_M}\mathbb{E}_{h \in \Phi_M-h'} e(\alpha_{\underline{n}}(h)) \cdot 1_{\Lambda_N}(\underline{n}) \int_X \Delta_{(\underline{n},n_{s+1})} \tilde{f}_l \ d\mu\right)>0.\]
		By subadditivity of $\limsup$ (and the pigeonhole principle), we can choose $h'_M \in G$ such that
		\begin{equation}\label{limsupequation} \limsup_{N \to \infty} \mathrm{Re} \left( \mathbb{E}_{\underline{n} \in \Phi_N^{s}} \mathbb{E}_{h \in \Phi_M-h'_M} e(\alpha_{\underline{n}}(h)) \cdot 1_{\Lambda_N}(\underline{n}) \int_X \Delta_{(\underline{n},n_{s+1})} \tilde{f}_l \ d\mu\right)>0.
		\end{equation}
		Let $(H_M)$ be the F\o lner sequence defined by $H_M:=\Phi_M-h'_M$. Then, \ref{limsupequation} implies
		\[ \liminf_{M \to \infty}\limsup_{N \to \infty} \mathrm{Re} \left( \mathbb{E}_{\underline{n} \in \Phi_N^{s}} \mathbb{E}_{n_{s+1} \in H_M} c_N(\underline{n},n_{s+1}) \cdot 1_{\Lambda_N}(\underline{n}) \int_X \Delta_{(\underline{n},n_{s+1})} \tilde{f}_l \ d\mu\right)>0,\]
		where
		\begin{equation}\label{lastmainprop}
			c_N(\underline{n},n_{s+1}):=e(\alpha_{\underline{n}}(n_{s+1})), \quad \underline{n} \in \Phi_N^s, n_{s+1} \in H_M, N, M \in \mathbb{N}.
		\end{equation}
		Hence, by the Cauchy--Schwarz inequality we deduce that
		\[ \liminf_{M \to \infty}\limsup_{N \to \infty} \norm{\mathbb{E}_{\underline{n} \in \Phi_N^{s}} \mathbb{E}_{n_{s+1} \in H_M} c_N(\underline{n},n_{s+1}) \cdot 1_{\Lambda_N}(\underline{n}) \cdot \Delta_{(\underline{n},n_{s+1})} \tilde{f}_l}_{L^2(\mu)}>0.\]
		Using \ref{characterequalitylast} and \ref{lastmainprop} we get that for all $N \in \mathbb{N}$ and $\underline{n} \in \Lambda_N$ we have
		\[ c_N(\underline{n},n_{s+1})=\prod_{j=1}^s b_{j,N}(\underline{n},n_{s+1}), \quad \underline{n} \in \Lambda_N, n_{s+1} \in H_M, N, M \in \mathbb{N},\]
		where $b_{j,N}$ do not depend on the variable $n_j$, for $j=1,\dots,s, N \in \mathbb{N}$, and are bounded by $1$. Since it is also the case that $1_{\Lambda_N}(\underline{n})$ does not depend on the variable $n_{s+1}$, we deduce from Lemma \ref{stepreduction} that
		\[ \nnorm{\tilde{f}_{l}}_{s+1}>0,\]
		completing the proof.
	\end{proof}
	We are now ready to give a proof of Theorem~\ref{framainthm}.
	\begin{proof}[Proof of Theorem~\ref{framainthm}] Fix an ergodic system $(X,\mu,(T_g)_{g \in G})$, a positive integer $\ell \in \mathbb{N}$, and mappings $a_1,\dots,a_{\ell}: G \to G$.
	
	First suppose that $\ell = 1$. The assumption that the sequence $a_1$ is good for equidistribution for the system $(X,\mu,(T_g)_{g\in G})$ implies property $(P_0)$ in the formulation of Theorem~\ref{framainthm} given at the beginning of this subsection. Next, still assuming $\ell = 1$, we show that property $(P_0)$ implies property $(P_1)$. Assume to the contrary that property $(P_0)$ holds and there exists $f_1 \in L^\infty(\mu)$ such that $\int f_1 \ d\mu = 0$ but
\begin{equation*}
\limsup_{N\to\infty} \norm{\mathbb{E}_{n \in \Phi_N} T_{a_1(n)} f_1}_{L^2(\mu)} \ > \ 0.
\end{equation*}
By Proposition~\ref{flprop}, there exists $\tilde{f}_1 \in L^\infty(\mu)$ of the form \ref{fldef} satisfying $\int \tilde{f}_1 \ d\mu = 0$ and
\begin{equation*}
\limsup_{N\to\infty} \norm{\mathbb{E}_{n \in \Phi_N} T_{a_1(n)}\tilde{f}_1}_{L^2(\mu)} \ > \ 0.
\end{equation*}
Since the sequence $a_1$ is good for seminorm estimates for $(X,\mu,(T_g)_{g\in G})$, we deduce that $\nnorm{\tilde{f}_1}_s > 0$ for some $s \geq 1$. If $s = 1$, by ergodicity we have $\int_X \tilde{f}_l \ d\mu \neq 0$, a contradiction. Otherwise, $s \geq 2$. Then, since the assumptions of Proposition \ref{mainprop} are satisfied, we deduce that $\int \tilde{f}_1 \ d\mu \neq 0$, the same contradiction.

	Now suppose that $\ell \geq 2$. We show that property $(P_l)$ holds by induction on $l \in \{0,\dots,\ell\}$.

		Assume first that $l=0$. In this case, $T_gf_j=e(\alpha_j(g))f_j$ for some $\alpha_j: G \to \mathbb{T}$, $j=1,\dots,\ell$. If $\alpha_j=0$ for $j=1,\dots,\ell$, then by ergodicity, $T_gf_j=\int_X f_j \ d\mu$, in which case \ref{mainthmeq} is obvious. Thus, to see that \ref{mainthmeq} holds, it suffices to show that
		\[ \lim_{N \to \infty} \mathbb{E}_{n \in \Phi_N} e(\alpha_1(a_1(n))+\dots+\alpha_{\ell}(a_{\ell}(n)))=0\]
		for $\alpha_1,\dots,\alpha_{\ell} \in \textrm{Spec}((T_g)_{g \in G})$, not all of them zero. This holds by the equidistribution assumption on $a_1,\dots,a_{\ell}$.

		For $l \in \{1,\dots,\ell\}$, we assume that the property $(P_{l-1})$ of Theorem \ref{framainthm} holds. We will show that property $(P_l)$ holds. Note that it suffices to assume that at least one of the functions $f_j$, $j \in \{1,\dots,\ell\} \setminus \{l\},$ has zero integral, by adding and subtracting the respective integrals for each $f_j$. The only noteworthy term is of the form $\prod_{j \in \{1,\dots,\ell\} \setminus \{l\}} \int_X f_j \ d\mu \cdot \mathbb{E}_{n \in \Phi_N} T_{a_l(n)}f_l$, which by the $\ell=1$ case converges in $L^2(\mu)$ to $\prod_{j \in \{1,\dots,\ell\}} \int_X f_j \ d\mu$.

		Our goal is to show that
		\begin{equation*}
			\lim_{N \to \infty} \mathbb{E}_{n \in \Phi_N} T_{a_1(n)}f_1 \cdots T_{a_{\ell}(n)}f_{\ell}=0,
		\end{equation*}
		where convergence takes place in $L^2(\mu)$. Arguing by contradiction, suppose that
		\[ \limsup_{N\to\infty} \norm{\mathbb{E}_{n \in \Phi_N} T_{a_1(n)}f_1 \cdots T_{a_{\ell}(n)}f_{\ell}}_{L^2(\mu)} \ > \ 0.\]
		Using Proposition \ref{flprop}, the same holds replacing $f_l$ with the function $\tilde{f}_l$ defined by the weak limit in \ref{fldef}, namely
		\begin{equation}\label{flmainthm}
			\tilde{f}_{l}:=\lim_{k \to \infty} \mathbb{E}_{n \in \Phi_{N_k}} T_{-a_{l}(n)}g_k \cdot \prod_{j=1, j \neq l}^{\ell} T_{a_j(n)-a_l(n)}\bar{f_j},
		\end{equation}
		for some F\o lner subsequence $(\Phi_{N_k})$, $g_k \in L^{\infty}(\mu)$, $k \in \mathbb{N}$, with all the functions bounded by $1$.
		
Since $f_{l+1},\dots,f_{\ell} \in \mathcal{E}((T_g)_{g \in g})$ and the mappings $a_1,\dots,a_{\ell}$ are good for seminorm estimates for the system $(X,\mu,(T_g)_{g \in G})$, we deduce 
		that $\nnorm{\tilde{f}_l}_s>0$ for some $s \in \mathbb{N}$. If $s = 1$, by ergodicity we have $\int_X \tilde{f}_l \ d\mu \neq 0$. Otherwise, $s \geq 2$. Then, since the assumptions of Proposition \ref{mainprop} are satisfied, we again have $\int_X \tilde{f}_l \ d\mu \neq 0$.

		In either case, using \ref{flmainthm}, one obtains
		\[ \lim_{k \to \infty} \mathbb{E}_{n \in \Phi_{N_k}} \int_X T_{-a_l(n)}g_k \cdot \prod_{j \in \{1,\dots,\ell\}, j \neq l} T_{a_j(n)-a_l(n)} \bar{f}_j \ d\mu \neq 0.\]
		Composing with $T_{a_l(n)}$ and using the Cauchy--Schwarz inequality, we get that
		\begin{equation}\label{eqn: some more garbage} \limsup_{k\to\infty} \norm{\mathbb{E}_{n \in \Phi_{N_k}} \prod_{j \in \{1,\dots,\ell\}, j \neq l} T_{a_j(n)} f_j}_{L^2(\mu)} \ > \ 0.\end{equation}
	Since at least one of the functions $f_j$, for $j \in \{1,\dots,\ell\} \setminus \{l\},$ has zero integral, and $f_{l+1},\dots,f_{\ell} \in \mathcal{E}((T_g)_{g \in G})$, using property $(P_{l-1})$ of Theorem \ref{framainthm} (with $f_l:=1\in \mathcal{E}((T_g)_{g \in G}))$, we would find that the limit of the left-hand side of \ref{eqn: some more garbage} is zero, a contradiction. We conclude that property $(P_l)$ holds.
	\end{proof}
	\section{Joint ergodicity of independent polynomial field actions}\label{sec: 4}
	Throughout this section, we will work with fields $F$ with characteristic zero.
	\subsection{Equidistribution}
	The purpose of this short subsection is to prove the following result:
	\begin{theorem}\label{characterequidistribution}
		Let $F$ be a countable field with characteristic zero. Let $\chi_1,\dots,\chi_k \in \hat{F}$ not all trivial. Let $(\Phi_N)$ be a F\o lner sequence in $F$ and let $p_1,\dots,p_k \in F[x]$ be linearly independent polynomials with $p_i(0)=0$, $i=1,\dots,k$. Then
		\begin{equation}\label{qcharacteravg}
			\lim_{N \to \infty} \mathbb{E}_{n \in \Phi_N} \chi_1(p_1(n))\cdots \chi_k(p_k(n))=0.
		\end{equation}
	\end{theorem}
	\begin{proof}
		We may assume without loss of generality that there are no non-trivial relations between the characters $\chi_i$. Indeed, suppose such a relation existed, so that there exist distinct $r_1,\ldots, r_t, s \in \{1,\ldots, k\}$ and (not necessarily distinct) elements $a_{r_1},\ldots,a_{r_t}, a_s \in F$, $a_s \neq 0$, such that
		\[ \chi_{r_1}(a_{r_1}n)\cdots \chi_{r_t}(a_{r_t}n)=\chi_s(a_sn)\]
		holds for all $n \in F$.
		Clearly, we may as well assume that not all of $a_{r_1},\dots,a_{r_t}$ are zero, for otherwise, $\chi_s(p_s(n))=\chi_s\left(a_s\cdot\frac{p_s(n)}{a_s}\right)=1$, which would imply that we could eliminate the character $\chi_s$ from the averages we are considering. Thus, we may assume that not all of $a_{r_1},\dots,a_{r_t}$ are zero. Then, we notice that
		\[ \chi_s(p_s(n))=\chi_s\left(a_s\cdot\frac{p_s(n)}{a_s}\right)=\chi_{r_1}\left(\frac{a_{r_1}p_s(n)}{a_s}\right)\dots\chi_{r_t}\left(\frac{a_{r_t}p_s(n)}{a_s}\right),\]
		which implies that the averages in \ref{qcharacteravg} can be simplified to
		\[  \mathbb{E}_{n \in \Phi_N} \prod_{i \neq s} \chi_i(\tilde{p}_i(n)),\]
		where $\tilde{p}_i:=p_i$ if $i \notin \{r_1,\dots,r_t\}$ and $\tilde{p}_i:=p_i+\frac{a_i}{a_s}p_s$ if $i \in \{r_1,\dots,r_t\}$. 
		It is straightforward to check that the family of polynomials $\tilde{p}_i$ is linearly independent, because the family of $p_i$ is.

		Thus, it is enough to show that \ref{qcharacteravg} holds under the extra assumption that the characters $\chi_i$ do not satisfy any non-trivial relation.
		
		After reordering and relabeling if necessary, we may assume that for some $t \in \{1,\dots,k\}$, $p_t,\dots,p_k$ are all of degree $d:=\max_{1\leq i \leq k} \deg p_i$. Applying \cite[Theorem~2.12]{bm}, another convenient form of the van der Corput trick, $d-1$ times, we see that it suffices to show that for all $h_1,\dots,h_{d-1} \in F\setminus \{0\}$ we have
		\begin{equation}\label{qcharacteravg2}
			\lim_{N \to \infty} \mathbb{E}_{n \in \Phi_N} \chi_t(\Delta_{h_1}\dotso\Delta_{h_{d-1}}p_t(n))\cdots \chi_k(\Delta_{h_1}\dotso\Delta_{h_{d-1}}p_k(n))=0,
		\end{equation}
		where, given $p \in F[x]$ and $h \in F \setminus \{0\}$, we put $(\Delta_hp)(n):=p(n+h)-p(n)$. Since $F$ has characteristic 0, we can find $a_{h_1,\dots,h_{d-1}; i} \in F \setminus\{0\}$, $b_{h_1,\dots,h_{d-1}; i} \in F$ such that $\Delta_{h_1}\dotso\Delta_{h_{d-1}}p_i(n)=a_{h_1,\dots,h_{d-1}; i}n+b_{h_1,\dots,h_{d-1}; i}$. Thus, to show \ref{qcharacteravg2} holds, it is enough to check that
		\[ \lim_{N \to \infty} \mathbb{E}_{n \in \Phi_N} \chi_t(a_{h_1,\dots,h_{d-1}; t}n)\cdots \chi_k(a_{h_1,\dots,h_{d-1}; k}n)=0. \]
		Since the characters under consideration do not satisfy any non-trivial relation, it is clearly enough to check that for any $\chi \in \hat{F} \setminus\{1\}$ we have
		\[ \lim_{N \to \infty} \mathbb{E}_{n \in \Phi_N}\chi(n)=0.\]
		This is straightforward: suppose that $z$ is a limit point of the sequence $(\mathbb{E}_{n \in \Phi_N}\break \chi(n))_{N \in \mathbb{N}}$. Since $\chi$ is a non-trivial character, there exists $q \in F$ such that $\chi(q) \neq 1$. Suppose that
		\[ \lim_{k \to \infty} \mathbb{E}_{n \in \Phi_{N_k}}\chi(n)=z.\]
		Since $(\Phi_{N_k})$ is still a F\o lner sequence, we have that
		\[ z=\lim_{k \to \infty}\mathbb{E}_{n \in \Phi_{N_k}}\chi(n+q)=z\chi(q).\]
		Given that $\chi(q)\neq 1$, we see that $z=0$. The sequence $(\mathbb{E}_{n \in \Phi_N}\chi(n))_{N \in \mathbb{N}}$ is in the closed disk of radius $1$ in $\mathbb{C}$, which is compact. This implies that at least one limit point must exist for $(\mathbb{E}_{n \in \Phi_N}\chi(n))_{N \in \mathbb{N}}$, and it must be $0$. No other limit point is possible, so the sequence must converge to $0$, completing the proof.
	\end{proof}
	\subsection{Seminorm estimates}
	The purpose of this subsection is to show that families of independent polynomials $p_1,\dots,p_k \in F[x]$ are good for seminorm estimates, one of the key requirements of Theorem~\ref{framainthm}. Thus, we want to show the following theorem, which proves a slightly stronger property.
	\begin{theorem}\label{BTZ factor is characteristic} Let $F$ be a countable field with characteristic zero. Let $d \in \mathbb{N}$. For any $r, b \in \mathbb{N}$, there exists $k \in \mathbb{N}$ such that for any family of nonconstant, essentially distinct\footnote{We say that a family of polynomials $p_1,\dots,p_r \in F[x_1,\dots,x_d]$ is \emph{essentially distinct} if $p_i-p_j$ is a non-constant polynomial for $i \neq j$.} polynomials $p_1, \ldots, p_r : F^d \to F$ of degree less than $b$ and any $f_1, \ldots, f_r \in L^\infty(X)$ with $\nnorm{f_1}_k = 0$, one has
		\begin{equation*} \lim_{N\to\infty} \mathbb{E}_{\mathbf{g}\in\Phi_N} T_{p_1(\mathbf{g})} f_1 \cdots T_{p_r(\mathbf{g})} f_r = 0
		\end{equation*}
		in $L^2(X)$ for any F\o lner sequence $(\Phi_N)$ in $(F^d,+)$.
	\end{theorem}
	\begin{remark} We will only care about $d = 1$, but the method we use here, introduced in \cite{leibmanconvergence}, makes use of the ``dimension increment trick'', so we will show the result for $F^d$.
	\end{remark}
	
	A stepping stone in proving Theorem~\ref{BTZ factor is characteristic} is the following proposition.
	\begin{proposition} \label{linear inequality for char thm} Let $F$ be a countable field with characteristic zero. Let $d \in \mathbb{N}$. Let $p_1, \ldots, p_r : F^d \to F$ be essentially distinct polynomials of degree one\footnote{We use degree to mean total degree.}. Then there exists a constant $C$ such that for any $f_1, \ldots, f_r \in L^\infty(X)$, we have
		\begin{equation*} \limsup_{N \to \infty} \norm{\mathbb{E}_{\mathbf{g}\in\Phi_N} T_{p_1(\mathbf{g})}f_1\cdots T_{p_r(\mathbf{g})}f_r}_{L^2(X)} \leq C\nnorm{f_1}_{r} \prod_{i=2}^r \norm{f_i}_{L^\infty(X)}
		\end{equation*}
		for any F\o lner sequence $(\Phi_N)$ in $F^d$.
	\end{proposition}
	\begin{corollary} \label{linear case of char thm} With assumptions as in Proposition \ref{linear inequality for char thm}, if additionally $\nnorm{f_1}_{r+1} = 0$, then we have
		\begin{equation*} \lim_{N\to\infty} \mathbb{E}_{\mathbf{g}\in\Phi_N} T_{p_1(\mathbf{g})}f_1\cdots T_{p_r(\mathbf{g})}f_r = 0
		\end{equation*}
		in $L^2(X)$ for any F\o lner sequence $(\Phi_N)$ in $F^d$.
	\end{corollary}
	We prove two lemmas before showing Proposition \ref{linear inequality for char thm}.
	\begin{lemma}\label{easy linear bound} Let $F$ be a countable field with characteristic zero. Let $d \in \mathbb{N}$. Let $p : F^d \to F$ be a degree one polynomial. Then for any $f \in L^\infty(X)$, we have
		\begin{equation*} \lim_{N\to\infty} \norm{\mathbb{E}_{\mathbf{g}\in\Phi_N} T_{p(\mathbf{g})}f}_{L^2(X)} \leq \nnorm{f}_1
		\end{equation*}
		for any F\o lner sequence $(\Phi_N)$ in $F^d$.
	\end{lemma}
	\begin{proof} In coordinates, we write $p(\mathbf{g}) = p(g_1,\ldots, g_d) = a_1g_1 + \cdots + a_dg_d + a_0$ for $a_0, a_1, \ldots, a_d \in F$, not all zero. Since $\nnorm{T_{a_0}f}_1 = \nnorm{f}_1$, we may replace $f$ by $T_{a_0}f$; thus assume $a_0 = 0$. 
		Since $F$ is a countable field, $F^d$ is a countable amenable group under addition. Since $p$ is linear, we may define a group action of $(F^d,+)$ on $L^2(X)$ by $U_\mathbf{g}(f') := T_{p(\mathbf{g})}f'$ for each $f' \in L^2(X)$ and $\mathbf{g} \in F^d$. By the mean ergodic theorem (Theorem~\ref{meanergthm}),
		\begin{equation*}
			\lim_{N\to\infty} \mathbb{E}_{\mathbf{g}\in\Phi_N} T_{p(\mathbf{g})}f = P_1f
		\end{equation*}
		in norm, where $P_1$ projects onto $\{ f' : T_{a_1g_1 + \cdots + a_dg_d}f'=f' \text{ for all } g_1,\ldots,g_d \in F\}$, which we observe equals $\{ f' : T_g f' = f' \text{ for all } g \in F\}$ since $F$ is a field. Thus, we have by the mean ergodic theorem that, for any F\o lner sequence $(\Psi_N)$ in $F$,
		\begin{align*}
			\lim_{N\to\infty} \norm{\mathbb{E}_{\mathbf{g}\in\Phi_N} T_{p(\mathbf{g})}f }_{L^2(X)}^2 \ & = \ \norm{P_2f}_{L^2(X)}^2 \\
			& = \ \left\langle P_2f,P_2f \right\rangle \\
			& = \ \left\langle f,P_2f \right\rangle \\
			& = \  \lim_{N\to\infty} \mathbb{E}_{g\in\Psi_N} \int_X f \cdot T_g \bar{f} \ d\mu  \\
			& = \ \lim_{N \to \infty} \mathbb{E}_{g \in \Psi_N}\nnorm{\Delta_g f}_0 \ = \ \nnorm{f}_1^2,
		\end{align*}
		proving the first lemma.
	\end{proof}
	\begin{lemma}\label{linear seminorm identity} Let $F$ be a countable field with characteristic zero. Let $d \in \mathbb{N}$. Let $p : F^d \to F$ be a degree one polynomial with $p(\mathbf{0}) = 0$. For any $f \in L^\infty(X)$, for any integer $k \geq 0$, and any F\o lner sequence $(\Phi_N)$ in $F^d$, we have
		\begin{equation*}
			\lim_{N\to\infty} \mathbb{E}_{\mathbf{g}\in\Phi_N} \nnorm{f \cdot T_{p(\mathbf{g})} \bar{f}}_k^{2^k} = \nnorm{f}_{k+1}^{2^{k+1}}.
		\end{equation*}
	\end{lemma}
	\begin{proof} The main idea in this lemma is the same as in Lemma \ref{easy linear bound}, namely that two apparently different invariant subspaces turn out to be the same because $F$ is a field. Unfolding definitions, letting $(\Psi_N)$ be a F\o lner sequence in $F$, and agreeing to discuss the marked equality afterwards, we calculate
		\begin{align*}
			& \lim_{N\to\infty} \mathbb{E}_{\mathbf{g}\in\Phi_N} \nnorm{f \cdot T_{p(\mathbf{g})} \bar{f}}_k^{2^k} \\ & = \ \lim_{N\to\infty} \mathbb{E}_{\mathbf{g}\in\Phi_N} \int_{X^{[k]}} \bigotimes_{\varepsilon \in \{0,1\}^k} \mathcal{C}^{|\varepsilon|}(f\cdot T_{p(\mathbf{g})} \bar{f}) \ d\mu^{[k]} \\
			& = \ \lim_{N\to\infty} \mathbb{E}_{\mathbf{g}\in\Phi_N} \int_{X^{[k]}} \left( \bigotimes_{\varepsilon \in \{0,1\}^k} \mathcal{C}^{|\varepsilon|}f \right) \cdot \left( \bigotimes_{\varepsilon \in \{0,1\}^k} T_{p(\mathbf{g})}\mathcal{C}^{|\varepsilon|}\bar{f} \right) d\mu^{[k]} \\
			& = \ \lim_{N\to\infty} \mathbb{E}_{\mathbf{g}\in\Phi_N} \int_{X^{[k]}} \left( \bigotimes_{\varepsilon \in \{0,1\}^k} \mathcal{C}^{|\varepsilon|}f \right) \cdot T_{p(\mathbf{g})}^{[k]} \left( \bigotimes_{\varepsilon \in \{0,1\}^k} \mathcal{C}^{|\varepsilon|}\bar{f} \right) d\mu^{[k]} \\
			& \overset{*}{=} \ \lim_{N\to\infty} \mathbb{E}_{g\in\Psi_N} \int_{X^{[k]}}\left( \bigotimes_{\varepsilon \in \{0,1\}^k} \mathcal{C}^{|\varepsilon|}f \right) \cdot T_{g}^{[k]} \left( \bigotimes_{\varepsilon \in \{0,1\}^k} \mathcal{C}^{|\varepsilon|}\bar{f} \right) d\mu^{[k]} \\
			& = \ \lim_{N\to\infty} \mathbb{E}_{g\in\Psi_N} \nnorm{f \cdot T_g \bar{f}}_k^{2^k} \\
			& = \ \nnorm{f}_{k+1}^{2^{k+1}}.
		\end{align*}
		For the marked equality, consider the following. Since $p$ is linear, we may define a group action of $(F^d, +)$ on $L^2(X^{[k]})$ by $U_\mathbf{g}(f') := T_{p(\mathbf{g})}^{[k]}f'$ for each $f' \in L^2(X^{[k]})$ and $\mathbf{g} \in F^d$. By the mean ergodic theorem (Theorem~\ref{meanergthm}),
		\begin{multline}\label{one}
			\lim_{N\to\infty} \mathbb{E}_{\mathbf{g}\in\Phi_N} \int_{X^{[k]}} \left( \bigotimes_{\varepsilon \in \{0,1\}^k}\mathcal{C}^{|\varepsilon|} f \right) \cdot T_{p(\mathbf{g})}^{[k]} \left( \bigotimes_{\varepsilon \in \{0,1\}^k} \mathcal{C}^{|\varepsilon|}\bar{f} \right) d\mu^{[k]} = \\ \int_{X^{[k]}} \left( \bigotimes_{\varepsilon \in \{0,1\}^k} \mathcal{C}^{|\varepsilon|}f \right) \cdot P_1 \left( \bigotimes_{\varepsilon \in \{0,1\}^k} \mathcal{C}^{|\varepsilon|}\bar{f} \right) d\mu^{[k]},
		\end{multline}
		where $P_1$ is the orthogonal projection onto the subspace $\{ f' : T_{p(\mathbf{g})}^{[k]} f' = f' \text{ for all } \mathbf{g} \in F^d \}$, which we observe equals $\{ f' : T_g^{[k]} f' = f' \text{ for all } g \in F\}$. Thus, since $U_g(f') := T_g^{[k]}f'$ for each $f' \in L^2(X^{[k]})$ and $g \in F$ defines a group action of $(F,+)$ on $L^2(X^{[k]})$, we see that
		\begin{multline}\label{two}
			\lim_{N\to\infty} \mathbb{E}_{g\in\Psi_N} \int_{X^{[k]}}\left( \bigotimes_{\varepsilon \in \{0,1\}^k} \mathcal{C}^{|\varepsilon|}f \right) \cdot T_{g}^{[k]} \left( \bigotimes_{\varepsilon \in \{0,1\}^k} \mathcal{C}^{|\varepsilon|}\bar{f} \right) d\mu^{[k]} = \\ \int_{X^{[k]}}\left( \bigotimes_{\varepsilon \in \{0,1\}^k} \mathcal{C}^{|\varepsilon|}f \right) \cdot P_2 \left( \bigotimes_{\varepsilon \in \{0,1\}^k} \mathcal{C}^{|\varepsilon|}\bar{f} \right) d\mu^{[k]},
		\end{multline}
		where $P_2$ is the orthogonal projection onto the subspace $\{ f' : T_g^{[k]} f' = f' \text{ for all } g \in F\}$. Since $P_1 = P_2$, the marked equality follows from \ref{one} and \ref{two}, proving the second lemma.
	\end{proof}
	\begin{proof}[Proof of Proposition \ref{linear inequality for char thm}] We induct on $r$. The case $r = 1$ follows from Lemma \ref{easy linear bound}. Let $r \geq 2$ and assume that for any $p_1, \ldots, p_{r-1} : F^d \to F$ essentially distinct polynomials of degree one there exists $C$ such that for any $f_1, \ldots, f_{r-1} \in L^\infty(X)$ we have
		\begin{equation*}
			\limsup_{N \to \infty} \norm{\mathbb{E}_{\mathbf{g}\in\Phi_N} T_{p_1(\mathbf{g})}f_1\cdots T_{p_{r-1}(\mathbf{g})}f_{r-1}}_{L^2(X)} \ \leq \ C\nnorm{f_1}_{r-1} \prod_{i=2}^{r-1} \norm{f_i}_{L^\infty(X)}
		\end{equation*}
		for any F\o lner sequence $(\Phi_N)$ in $F^d$.
		
		Let $p_1, \ldots, p_r : F^d \to F$ be essentially distinct polynomials of degree one, and let $(\Phi_N)$ be a F\o lner sequence in $F^d$. We would want to show there exists $C$ such that for any $f_1, \ldots, f_{r} \in L^\infty(X)$ we have
		\begin{equation*}
			\limsup_{N \to \infty} \norm{\mathbb{E}_{\mathbf{g}\in\Phi_N} T_{p_1(\mathbf{g})}f_1\cdots T_{p_{r}(\mathbf{g})}f_{r}}_{L^2(X)} \ \leq \ C\nnorm{f_1}_{r} \prod_{i=2}^{r} \norm{f_i}_{L^\infty(X)}.
		\end{equation*}
		It suffices to consider only the case when $p_i(\mathbf{0}) = 0$ and $||f_i||_{L^\infty(X)} \leq 1$ for all $i \in \{1, \ldots, r\}$. However, note that under the latter assumption, it actually suffices to show there exists $C$ such that for any $f_1, \ldots, f_{r} \in L^\infty(X)$ we have
		\begin{equation*}
			\limsup_{N \to \infty} \norm{\mathbb{E}_{\mathbf{g}\in\Phi_N} T_{p_1(\mathbf{g})}f_1\cdots T_{p_{r}(\mathbf{g})}f_{r}}_{L^2(X)} \ \leq \ C\nnorm{f_1}_{r},
		\end{equation*}
		so let us show this.
		
		Applying part (1) of Lemma \ref{2D trick} with $x_\mathbf{g} = T_{p_1(\mathbf{g})}f_1\cdots T_{p_r(\mathbf{g})}f_r \in L^2(X)$ for $\mathbf{g} \in F^d$, exploiting the $F$-invariance of $\mu$ and the linearity of our $p_i$'s, and using Cauchy--Schwarz, we see that for any finite subset $S \subset F^d$,
		\begin{align*}
			& \limsup_{N \to \infty} \norm{\mathbb{E}_{\mathbf{g}\in\Phi_N} T_{p_1(\mathbf{g})}f_1\cdots T_{p_{r}(\mathbf{g})}f_{r}}_{L^2(X)}^2 \\
			& \leq \ \limsup_{N \to \infty} \mathbb{E}_{(\mathbf{h},\mathbf{h'})\in S^2} \mathbb{E}_{\mathbf{g}\in\Phi_N} \int_X \prod_{i=1}^r T_{p_i(\mathbf{g} + \mathbf{h})}f_i\prod_{i=1}^r T_{p_i(\mathbf{g} + \mathbf{h'})}\bar{f}_i \ d\mu \\
			& = \ \limsup_{N \to \infty} \mathbb{E}_{(\mathbf{h},\mathbf{h'})\in S^2} \mathbb{E}_{\mathbf{g}\in\Phi_N} \\ & \ \int_X \left( \prod_{i=1}^{r-1} T_{(p_i-p_r)(\mathbf{g})}\left( T_{p_i(\mathbf{h})}f_i\cdot T_{p_i(\mathbf{h'})}\bar{f}_i\right) \right) \left( T_{p_r(\mathbf{h})}f_r\cdot T_{p_r(\mathbf{h'})}\bar{f}_r \right) d\mu \\
			& \leq \ \mathbb{E}_{(\mathbf{h},\mathbf{h'})\in S^2} \limsup_{N \to \infty} \norm{\mathbb{E}_{\mathbf{g}\in\Phi_N} \prod_{i=1}^{r-1} T_{(p_i-p_r)(\mathbf{g})}(T_{p_i(\mathbf{h})} f_i \cdot T_{p_i(\mathbf{h'})}\bar{f}_i)}_{L^2(X)}.
		\end{align*}
		By the inductive hypothesis, there exists $C'$ (independent of the $f_i$'s and of $(\Phi_N)$) such that for any $\mathbf{h}, \mathbf{h'} \in F^d$, we have
		\begin{align*}
			& \limsup_{N\to\infty} \norm{\mathbb{E}_{\mathbf{g}\in\Phi_N} \prod_{i=1}^{r-1} T_{(p_i-p_r)(\mathbf{g})}(T_{p_i(\mathbf{h})} f_i \cdot T_{p_i(\mathbf{h'})}\bar{f}_i)}_{L^2(X)} \\
			& \leq \ C' \nnorm{T_{p_1(\mathbf{h})} f_1 \cdot T_{p_1(\mathbf{h'})}\bar{f}_1}_{r-1} \prod_{i=2}^{r-1} \norm{T_{p_i(\mathbf{h})} f_i \cdot T_{p_i(\mathbf{h'})}\bar{f}_i}_{L^\infty(X)} \\
			& \leq \ C' \nnorm{T_{p_1(\mathbf{h})} f_1 \cdot T_{p_1(\mathbf{h'})}\bar{f}_1}_{r-1}.
		\end{align*}
		Thus, after combining the above displays and applying $F$-invariance of $\nnorm{\cdot}_{r-1}$ and Jensen's inequality, we get for any finite subset $S \subset F^d$ that
		\begin{align*}
			& \limsup_{N \to \infty} \norm{\mathbb{E}_{\mathbf{g}\in\Phi_N} T_{p_1(\mathbf{g})}f_1\cdots T_{p_{r}(\mathbf{g})}f_{r}}_{L^2(X)} \\
			& \leq \ \left( C' \ \mathbb{E}_{(\mathbf{h},\mathbf{h'})\in S^2} \nnorm{T_{p_1(\mathbf{h})} f_1 \cdot T_{p_1(\mathbf{h'})}\bar{f}_1}_{r-1} \right)^{1/2} \\
			& = \ (C')^{1/2} \left( \mathbb{E}_{(\mathbf{h},\mathbf{h'})\in S^2} \nnorm{f_1 \cdot T_{p_1(\mathbf{h'}-\mathbf{h})}\bar{f}_1}_{r-1}\right)^{1/2} \\
			& \leq \ (C')^{1/2} \left( \mathbb{E}_{(\mathbf{h},\mathbf{h'})\in S^2} \nnorm{f_1 \cdot T_{p_1(\mathbf{h'}-\mathbf{h})}\bar{f}_1}_{r-1}^{2^{r-1}} \right)^{2^{-r}}.
		\end{align*}
		To finish, let $(\Psi_N)$ be a F\o lner sequence in $F^d$. Then $(\Psi_N \times \Psi_N)$ is a F\o lner sequence in $F^{2d}$ and $(\mathbf{h},\mathbf{h'}) \mapsto p(\mathbf{h},\mathbf{h'}) := p_1(\mathbf{h'}-\mathbf{h})$ is a degree one polynomial from $F^{2d}$ to $F$ with $p(\mathbf{0}) = 0$, so letting $S = \Psi_N$ successively for each $N$ and applying Lemma \ref{linear seminorm identity}, we conclude
		\begin{equation*}
			\limsup_{N \to \infty} \norm{\mathbb{E}_{\mathbf{g}\in\Phi_N} T_{p_1(\mathbf{g})}f_1\cdots T_{p_{r}(\mathbf{g})}f_{r}}_{L^2(X)} \ \leq \ (C')^{1/2} \nnorm{f_1}_{r},
		\end{equation*}
		completing the induction.
	\end{proof}
	
	We now prepare to execute the PET induction, a technique introduced in \cite{wmpet}. Here we follow the presentation of Leibman in \cite{leibmanconvergence}.
	
	Fix $d \in \mathbb{N}$. A \textit{system} $P$ is a finite collection of polynomials on $F^d$, whose \textit{degree} $\deg P$ is defined to be the maximum of the total degrees of these polynomials. A system $P = \{p_1, \ldots, p_r\}$ is \textit{standard} if all $p_i$ are nonconstant and essentially distinct, i.e. $p_i - p_j$ is not constant for $i \neq j$, and also $\deg p_1 = \deg P$. If $p$ and $q$ are polynomials on $F^d$, define $p$ to be equivalent to $q$ iff $\deg p = \deg q$ and $\deg (p-q) < \deg p$. This defines an equivalence relation on the set of all such polynomials, which thus partitions a given system $P$ into equivalence classes. The degree of an equivalence class is the common degree of its polynomials. For a system $P$, define its \textit{weight} $\omega (P)$ as the vector $(\omega_1, \ldots, \omega_{\deg P})$, where $\omega_i$ is the number of equivalence classes of degree $i$ in $P$.\footnote{We will not need to count the number of equivalence classes of degree 0.} For two given integer vectors $\omega = (\omega_1, \ldots, \omega_m)$ and $\omega ' = (\omega_1 ', \ldots, \omega_{m'} ')$, define $\omega < \omega '$ iff either condition
	\begin{align*}
		(1) & \quad m < m' \\
		(2) & \quad m = m' \text{ and there is an } n \leq m \text{ such that } \omega_n < \omega_n ' \text{ and } \omega_i = \omega_i ' \text{ for } n < i \leq m
	\end{align*}
	is met. This defines a well ordering on the set of weights of systems. We plan to induct on this well ordering of weights. One final agreement before we start: We say a property holds \textit{for almost all} $\mathbf{g} \in F^d$ if the set of elements of $F^d$ for which it does not hold is contained in the set of zeroes of a nontrivial polynomial on $F^d$. The reason for this is that the set of zeros of a nontrivial polynomial $p$, say $E$, has zero upper Banach density. Indeed, by equation (5) right after \cite[Remark~1.1]{bbf} to show that $d^*(E)=0$ it is enough to show that for any tuple of F\o lner sequences $(\Phi_{1,N},\dots,\Phi_{d,N})$ we have $\bar{d}_{(\Phi_{1,N} \times \dots \times \Phi_{d,N})}(E)=0$. This is straightforward: for fixed $(u_1,\dots,u_{d-1}) \in \Phi_{1,N} \times \dots \times \Phi_{d-1, N}$, the set of $v \in F$ such that $p(u_1,\dots,u_{d-1},v)=0$ has cardinality bounded above by the degree of the polynomial in $v$ (which is in turn bounded above by some constant $C$ independent of the choice of $u_1,\dots,u_{d-1}$). Thus, $\frac{|E \cap (\Phi_{1,N} \times \dots \times \Phi_{d,N})|}{|\Phi_{1,N}\times\dots\times \Phi_{d,N}|} \leq \frac{C}{|\Phi_{d,N}|},$ which goes to $0$ as $N \to \infty$.
	
	We now show the following.
	\begin{proposition} \label{BTZ factor is characteristic warmup} Let $F$ be a countable field with characteristic zero. Let $d \in \mathbb{N}$. For any $r \in \mathbb{N}$ and any integer vector $\omega = (\omega_1, \ldots, \omega_\ell)$, there is a $k \in \mathbb{N}$ such that for any standard system $P = \{p_1, \ldots, p_r : F^d \to F\}$ of weight $\omega$ and any $f_1, \ldots, f_r \in L^\infty(X)$ with $\nnorm{f_1}_k = 0$, we have
		\[ \lim_{N\to\infty} \mathbb{E}_{\mathbf{g}\in\Phi_N} T_{p_1(\mathbf{g} )} f_1 \cdots T_{p_r(\mathbf{g} )} f_r = 0 \]
		in $L^2(X)$ for any F\o lner sequence $(\Phi_N)$ in $F^d$.
	\end{proposition}
	\begin{proof} We proceed by induction on the weight of the standard system. The base case handles all weights corresponding to systems of degree one and follows by Corollary \ref{linear case of char thm}. Let $P = \{p_1, \ldots, p_r\}$ be a standard system with $\deg P \geq 2$ and weight $\omega$. By strong induction, there exists\footnote{There are finitely many possible weights $\omega '$ of systems of $s \leq 2r$ polynomials such that $\omega ' < \omega$, and the seminorms $\nnorm{\cdot}_k$ form a nondecreasing sequence.} $k \in \mathbb{N}$ such that for any standard system $\{ q_1, \ldots, q_s\}$ with $s \leq 2r$ with weight $\omega' < \omega$ and any $\tilde{f}_1, \ldots, \tilde{f}_s \in L^\infty(X)$ with $\nnorm{\tilde{f}_1}_k = 0$, we have
		\[ \lim_{N\to\infty} \mathbb{E}_{\mathbf{g}\in\Phi_N} T_{q_1(\mathbf{g})} \tilde{f}_1 \cdots T_{q_s(\mathbf{g})} \tilde{f}_s = 0 \]
		in $L^2(X)$ for any F\o lner sequence $(\Phi_N)$ in $F^d$.
		
		Choose $i_0 \in \{2, \ldots, r\}$ such that $p_{i_0}$ has the minimal degree in our given system $P$. In case all polynomials in $P$ have the same degree, if possible, choose $i_0$ to be such that $p_{i_0}$ is not equivalent to $p_1$. For each $\mathbf{h}, \mathbf{h'} \in F^d$, define the system
		\[ P_{\mathbf{h}, \mathbf{h'}} := \{ p_i(\mathbf{g} + \mathbf{h}), p_i(\mathbf{g} + \mathbf{h'}) : \deg p_i > 1 \} \cup \{ p_i(\mathbf{g} + \mathbf{h'}) : \deg p_i = 1 \},\]
		where $p_i(\mathbf{g} + \mathbf{h})$ and $p_i(\mathbf{g} + \mathbf{h'})$ are viewed as polynomials in $\mathbf{g}$. Order $P_{\mathbf{h}, \mathbf{h'}} = \{ q_{\mathbf{h}, \mathbf{h'}, 1} , \ldots, q_{\mathbf{h}, \mathbf{h'}, s}\}$ in some way so that $q_{\mathbf{h}, \mathbf{h'}, 1}(\mathbf{g}) = p_1(\mathbf{g} + \mathbf{h})$ and $q_{\mathbf{h}, \mathbf{h'}, s}(\mathbf{g}) = p_{i_0}(\mathbf{g} + \mathbf{h'})$. Then $P_{\mathbf{h}, \mathbf{h'}}$ is a standard system for almost all $(\mathbf{h}, \mathbf{h'}) \in F^{2d}$. Moreover, we have not changed the equivalence classes $P$ had: For any $\mathbf{h}, \mathbf{h'} \in F^d$ and $i \in \{1, \ldots, r\}$, the polynomials $p_i(\mathbf{g} + \mathbf{h})$ and $p_i(\mathbf{g} + \mathbf{h'})$ are equivalent to $p_i(\mathbf{g})$, so it follows that $\omega(P_{\mathbf{h}, \mathbf{h'}}) = \omega(P) = \omega$ always.
		
		Next, for any $\mathbf{h}, \mathbf{h'}$, define the system
		\[ P_{\mathbf{h}, \mathbf{h'}}' := \{ q_{\mathbf{h}, \mathbf{h'}, 1} - q_{\mathbf{h}, \mathbf{h'}, s}, \ldots, q_{\mathbf{h}, \mathbf{h'}, s-1} - q_{\mathbf{h}, \mathbf{h'}, s} \}. \]
		We will apply the induction hypothesis to these, so we check the following. For almost all $(\mathbf{h}, \mathbf{h'}) \in F^{2d}$, $P_{\mathbf{h}, \mathbf{h'}}$ is standard. This follows first since the nonconstant and essentially distinct nature of the $q_{\mathbf{h}, \mathbf{h'}, j}, j \in \{1, \ldots, s-1\}$ implies the same for $q_{\mathbf{h}, \mathbf{h'}, j} - q_{\mathbf{h}, \mathbf{h'}, s}, j \in \{1, \ldots, s-1\}$. Second, if $p_{i_0}$ is not equivalent to $p_1$, then $p_{i_0}$ has been chosen either to have smaller degree than $p_1$ or to have the same degree but not reduce the degree of $p_1$ on subtracting, so that $\deg (q_{\mathbf{h}, \mathbf{h'}, 1} - q_{\mathbf{h}, \mathbf{h'}, s}) = \deg p_1 = \deg P_{\mathbf{h}, \mathbf{h'}}$ for all $\mathbf{h}, \mathbf{h'}$; and if $p_{i_0}$ is equivalent to $p_1$, then $\deg (q_{\mathbf{h}, \mathbf{h'}, 1} - q_{\mathbf{h}, \mathbf{h'}, s}) = \deg p_1 - 1= \deg P_{\mathbf{h}, \mathbf{h'}}$ for almost all $\mathbf{h}, \mathbf{h'}$. Finally, we have reduced the weight: For all $(\mathbf{h}, \mathbf{h'}) \in F^{2d}$, we have $\omega (P_{\mathbf{h}, \mathbf{h'}}') < \omega$. This follows since the equivalence classes in $P_{\mathbf{h}, \mathbf{h'}}'$ and their degrees are the same as those in $P_{\mathbf{h}, \mathbf{h'}}$, with one exception: The class in $P_{\mathbf{h}, \mathbf{h'}}$ to which $q_{\mathbf{h}, \mathbf{h'}, s}$ belongs has either vanished or been split into some new classes of lesser degree.
		
		Now we can combine everything. Let $f_1, \ldots, f_r \in L^\infty(X)$ with $\nnorm{f_1}_k = 0$ and let $(\Phi_N)$ be a F\o lner sequence in $F^d$. We may assume that $||f_2||_{L^\infty(X)}, \ldots, ||f_r||_{L^\infty(X)} \leq 1$. By part (1) of Lemma \ref{2D trick} applied to $x_\mathbf{g} = \prod_{i=1}^r T_{p_i(\mathbf{g})} f_i \in L^2(X)$ for $\mathbf{g} \in F^d$, for any finite set $S \subset F^d$, we have
		\begin{align*}
			& \limsup_{N\to\infty} \norm{\mathbb{E}_{\mathbf{g}\in\Phi_N} \prod_{i=1}^r T_{p_i(\mathbf{g})} f_i}_{L^2(X)}^2 \\
			& \leq \ \limsup_{N\to\infty} \mathbb{E}_{(\mathbf{h},\mathbf{h'})\in S^2} \mathbb{E}_{\mathbf{g}\in\Phi_N} \int_X \prod_{i=1}^r T_{p_i(\mathbf{g}+\mathbf{h})} f_i \cdot \prod_{i=1}^r T_{p_i(\mathbf{g}+\mathbf{h'})} \bar{f}_i \ d\mu \\
			& = \ \limsup_{N\to\infty} \mathbb{E}_{(\mathbf{h},\mathbf{h'})\in S^2} \mathbb{E}_{\mathbf{g}\in\Phi_N} \\ & \ \int_X \prod_{\deg p_i > 1} \left( T_{p_i(\mathbf{g}+\mathbf{h})} f_i \cdot T_{p_i(\mathbf{g}+\mathbf{h'})} \bar{f}_i \right) \prod_{\deg p_i = 1} T_{p_i(\mathbf{g}+\mathbf{h'})} (f_i \cdot T_{p_i(\mathbf{h}) - p_i(\mathbf{h'})} \bar{f}_i)\ d\mu \\
			& = \ \limsup_{N\to\infty} \mathbb{E}_{(\mathbf{h},\mathbf{h'})\in S^2} \mathbb{E}_{\mathbf{g}\in\Phi_N} \int_X \prod_{j=1}^s T_{q_{\mathbf{h},\mathbf{h'},j}(\mathbf{g})} \tilde{f}_{\mathbf{h}, \mathbf{h'}, j} \ d\mu \\
			& = \ \limsup_{N\to\infty} \mathbb{E}_{(\mathbf{h},\mathbf{h'})\in S^2} \mathbb{E}_{\mathbf{g}\in\Phi_N} \int_X \left( \prod_{j=1}^{s-1} T_{(q_{\mathbf{h},\mathbf{h'},j} - q_{\mathbf{h}, \mathbf{h'}, s})(\mathbf{g})} \tilde{f}_{\mathbf{h}, \mathbf{h'}, j} \right) \cdot \tilde{f}_{\mathbf{h},\mathbf{h'}, s} \ d\mu \\
			& \leq \ \mathbb{E}_{(\mathbf{h},\mathbf{h'})\in S^2} \limsup_{N\to\infty} \norm{ \mathbb{E}_{\mathbf{g}\in\Phi_N} \prod_{j=1}^{s-1} T_{(q_{\mathbf{h},\mathbf{h'},j} - q_{\mathbf{h}, \mathbf{h'}, s})(\mathbf{g})} \tilde{f}_{\mathbf{h}, \mathbf{h'}, j}}_{L^2(X)},
		\end{align*}
		where the first equality is a rearrangement, the second equality follows on setting (for $\mathbf{h}, \mathbf{h'} \in F^d$) $q_{\mathbf{h}, \mathbf{h'}, 1}, \ldots, q_{\mathbf{h}, \mathbf{h'}, s}$ to be the polynomials of the system $P_{\mathbf{h}, \mathbf{h'}}$ and setting $\tilde{f}_{\mathbf{h}, \mathbf{h'}, j}$ to be either $f_i$ or $f_i \cdot T_{p_i(\mathbf{h})-p_i(\mathbf{h'})} \bar{f}_i$ for appropriate\footnote{To clarify, set the function $\tilde{f}_{\mathbf{h}, \mathbf{h'}, j}$ to be $f_i$ when $q_{\mathbf{h}, \mathbf{h'}, j}$ is of the form $p_i(\mathbf{g} + \mathbf{h})$ or $p_i(\mathbf{g} + \mathbf{h'})$ with $\deg p_i > 1$, otherwise set it to be $f_i \cdot T_{p_i(\mathbf{h})-p_i(\mathbf{h'})} \bar{f}_i$.} $i$, the third equality follows from invariance, and the last inequality is Cauchy--Schwarz. Note that since $\deg p_1 = \deg P \geq 2$, we have $\tilde{f}_{\mathbf{h}, \mathbf{h'}, 1} = f_1$. Thus, by the induction hypothesis applied to $P_{\mathbf{h}, \mathbf{h'}}'$, we have
		\[ \lim_{N\to\infty} \norm{\mathbb{E}_{\mathbf{g}\in\Phi_N} \prod_{j=1}^{s-1} T_{(q_{\mathbf{h},\mathbf{h'},j} - q_{\mathbf{h}, \mathbf{h'}, s})(\mathbf{g})} \tilde{f}_{\mathbf{h}, \mathbf{h'}, j} }_{L^2(X)} = 0 \]
		for almost all $(\mathbf{h}, \mathbf{h'}) \in F^{2d}$. For the remaining $(\mathbf{h}, \mathbf{h'}) \in F^{2d}$, the norm in the previous display is bounded by 1, which implies that
		\[ \inf_{S \subset F^d \text{ finite}} \mathbb{E}_{(\mathbf{h},\mathbf{h'})\in S^2} \limsup_{N\to\infty} \norm{ \mathbb{E}_{\mathbf{g}\in\Phi_N} \prod_{j=1}^{s-1} T_{(q_{\mathbf{h},\mathbf{h'},j} - q_{\mathbf{h}, \mathbf{h'}, s})(\mathbf{g})} \tilde{f}_{\mathbf{h}, \mathbf{h'}, j}}_{L^2(X)} = 0, \]
		by taking $F$ to be the successive terms of a F\o lner sequence. Combined with the previous displayed inequality, this completes the proof\footnote{It follows from the proof that the step of the seminorm $k$ depends only on the degree of the polynomials $p_1,\dots,p_r$ and $d$.}.
	\end{proof}
	
	\begin{proof}[Proof of Theorem \ref{BTZ factor is characteristic}] For standard systems, Proposition \ref{BTZ factor is characteristic warmup} does it. We reduce the nonstandard case to this one. Let $P = \{p_1, \ldots, p_r\}$ be a nonstandard system of nonconstant essentially distinct polynomials $F^d \to F$ of degree at most $b$, let $f_1, \ldots, f_r \in L^\infty(X)$ and let $(\Phi_N)$ be a F\o lner sequence in $F^d$. By part (2) of Lemma \ref{2D trick} applied to $x_\mathbf{g} = \prod_{i=1}^r T_{p_i(\mathbf{g})} f_i \in L^2(X)$ for $\mathbf{g} \in F^d$, there exists a F\o lner sequence $(\Theta_M)$ in $F^{3d}$ such that
		\begin{align*}
			& \limsup_{N \to \infty} \norm{\mathbb{E}_{\mathbf{g}\in\Phi_N} \prod_{i=1}^r T_{p_i(\mathbf{g})} f_i}^2_{L^2(X)} \\
			& \leq \ \limsup_{M \to \infty} \mathbb{E}_{(\mathbf{g},\mathbf{h},\mathbf{h'}) \in\Theta_M} \int \prod_{i=1}^r T_{p_i(\mathbf{g}+\mathbf{h})} f_i \cdot \prod_{i=1}^r T_{p_i(\mathbf{g}+\mathbf{h'})} \bar{f}_i \ d\mu \\
			& = \ \limsup_{M \to \infty} \mathbb{E}_{(\mathbf{g},\mathbf{h},\mathbf{h'}) \in\Theta_M} \int \prod_{i=1}^r T_{p_i(\mathbf{g}+\mathbf{h}) + q(\mathbf{g})} f_i \cdot \prod_{i=1}^r T_{p_i(\mathbf{g}+\mathbf{h'}) + q(\mathbf{g})} \bar{f}_i \ d\mu \\
			& \leq \limsup_{M \to \infty} \norm{\mathbb{E}_{(\mathbf{g},\mathbf{h},\mathbf{h'}) \in\Theta_M} \prod_{i=1}^r T_{p_i(\mathbf{g}+\mathbf{h}) + q(\mathbf{g})} f_i \cdot \prod_{i=1}^r T_{p_i(\mathbf{g}+\mathbf{h'}) + q(\mathbf{g})} \bar{f}_i }_{L^2(X)},
		\end{align*}
		where $q$ is any polynomial $F^d \to F$ of degree $b$. The system
		\[ \{ p_1(\mathbf{g} + \mathbf{h}) + q(\mathbf{g}), \ldots, p_r(\mathbf{g} + \mathbf{h}) + q(\mathbf{g}), p_1(\mathbf{g}+\mathbf{h'}) + q(\mathbf{g}), \ldots, p_r(\mathbf{g}+\mathbf{h'}) + q(\mathbf{g}) \}\]
		of polynomials $F^{3d} \to F$ is standard, of degree $b$, with $2r$ elements. Thus by Proposition \ref{BTZ factor is characteristic warmup} there exists $k \in \mathbb{N}$ depending on $r$ and $b$ such that
		\[ \lim_{M\to\infty} \mathbb{E}_{(\mathbf{g},\mathbf{h},\mathbf{h'}) \in\Theta_M} \prod_{i=1}^r T_{p_i(\mathbf{g}+\mathbf{h}) + q(\mathbf{g})} f_i \cdot \prod_{i=1}^r T_{p_i(\mathbf{g}+\mathbf{h'}) + q(\mathbf{g})} \bar{f}_i = 0 \]
		in $L^2(X)$ whenever $\nnorm{f_1}_k = 0$.
	\end{proof}
	\begin{remark}
		An alternative approach to proving Theorem~\ref{BTZ factor is characteristic} is to adapt the proof of \cite[Lemma~4.7]{fravariable} without the use of brackets. In doing so, we would not need to work with polynomials of several variables.
	\end{remark}
	\subsection{Proof of Theorem~\ref{findep}}
	We are now in a position to prove the main result of this paper:
\vspace{4pt}
	\begin{thmrep}[\ref{findep}]
		Let $F$ be a countable field with characteristic zero. Let $(X,\mu,\break (T_n)_{n \in F})$ be an ergodic $F$-system. Let $(\Phi_N)$ be a F\o lner sequence, let $k \in \mathbb{N}$, let $p_1,\dots,p_k \in F[x]$ be independent polynomials, and let $f_1,\dots,f_k \in L^{\infty}(\mu)$ be functions. Then
		\[ \lim_{N \to \infty} \mathbb{E}_{n \in \Phi_N} T_{p_1(n)}f_1\cdots T_{p_k(n)}f_k\ = \ \prod_{j=1}^k \int_X f_i \ d\mu\]
		in $L^2(\mu)$.
	\end{thmrep}
	\begin{proof}
		Without loss of generality we may assume that $p_i(0)=0$, $i=1,\dots,k$, since $\mu$ is invariant under $(T_n)_{n \in F}$. Note that a family of independent polynomials $p_1,\dots,p_k \in F[x]$ is essentially distinct. Apply Theorems~\ref{framainthm},~\ref{characterequidistribution},~and~\ref{BTZ factor is characteristic}.
	\end{proof}
	As a corollary, we can give a version of Theorem \ref{findep} for systems that are not necessarily ergodic:
	\begin{corollary}
		Let $F$ be a countable field with characteristic zero. Let $(X,\mu,\break (T_n)_{n \in F})$ be an $F$-system. Let $(\Phi_N)$ be a F\o lner sequence, let $k \in \mathbb{N}$, let $p_1,\dots,p_k \in F[x]$ be independent polynomials, and let $f_1,\dots,f_k \in L^{\infty}(\mu)$ be functions. Then
		\begin{equation}\label{fnonerg} \lim_{N \to \infty} \mathbb{E}_{n \in \Phi_N} T_{p_1(n)}f_1\cdots T_{p_k(n)}f_k\ = \ \prod_{j=1}^k \mathbb{E}[f_i|\mathcal{I}((T_n)_{n \in F})]
		\end{equation}
		in $L^2(\mu)$, where $\mathcal{I}((T_n)_{n \in F})$ is the $\sigma$-algebra of sets that are invariant under the $F$-action $(T_n)_{n \in F}$.
	\end{corollary}
	\begin{proof}
		Let $\mu=\int_X \mu_x \ d\mu(x)$ be the ergodic decomposition of $\mu$ with respect to the action $(T_n)_{n \in F}$, so $\mu_x$ is ergodic for $\mu$-a.e. $x \in X$. Recall that one has
		\[ \mathbb{E}[f_i| \mathcal{I}((T_n)_{n \in F})](x)=\int_X f_i \ d\mu_x,\]
		where the equality holds $\mu$-a.e. Therefore, it suffices to show that
		\begin{equation}\label{fnonerg2} \int_X \left| \mathbb{E}_{n \in \Phi_N} T_{p_1(n)}f_1\cdots T_{p_k(n)}f_k-\prod_{i=1}^k \int_X f_i \ d\mu_x \right|^2 \ d\mu(x) \to 0
		\end{equation}
		as $N \to \infty$. Note that using the ergodic decomposition, the convergence in \ref{fnonerg2} holds if and only if
		\begin{equation}\label{fnonerg3} \int_X\left(\int_X \left| \mathbb{E}_{n \in \Phi_N} T_{p_1(n)}f_1\cdots T_{p_k(n)}f_k-\prod_{i=1}^k \int_X f_i \ d\mu_x \right|^2 \ d\mu_x(x) \right)\ d \mu(x) \to 0.
		\end{equation}
		To see that \ref{fnonerg3} holds, apply Theorem~\ref{findep} to the term in parentheses in equation \ref{fnonerg3} to deduce (since $\mu$-a.e. measure $\mu_x$ is ergodic) that the aforementioned term (seen as a pointwise limit function of $x$) is equal to $0$ almost everywhere. Since it is bounded by assumption on the functions $f_i$, \ref{fnonerg3} follows from the dominated convergence theorem.
	\end{proof}
	\subsection{Corollaries}
	We now deduce some consequences of the results in the previous subsection.
\vspace{4pt}
	\begin{correp}[\ref{largesetspoly}]
		Let $F$ be a countable field with characteristic zero. Let $E \subseteq F$ with $d^*(E)>0$, $k \in \mathbb{N}$, and $p_1,\dots,p_k \in F[x]$ be independent polynomials. Then, for each $\varepsilon>0$, the set $\{ n \in F : d^*(E \cap (E-p_1(n)) \cap \dots \cap (E-p_k(n)))>d^*(E)^{k+1}-\varepsilon\}$ is syndetic.
	\end{correp}
	\begin{proof}
		This follows from an ergodic Furstenberg correspondence principle. Indeed, by \cite[Theorem~2.8]{bf} there exists an ergodic measure-preserving system $(X,\mathcal{B},\mu,\break (T_n)_{n \in F})$ and a set $A \in \mathcal{B}$ with $\mu(A)=d^*(E)$ such that for all $k \in \mathbb{N}$, $n_1,\dots,n_k \in F$, we have
		\[ d^*(E \cap (E-n_1) \cap \dots \cap (E-n_k)) \geq \mu(A \cap  T_{-n_1} A \cap \dots \cap T_{-n_k}A).\]
		By Theorem \ref{findep} it follows that for every F\o lner sequence $(\Phi_N)$ we have
		\[ \liminf_{N \to \infty} \mathbb{E}_{n \in \Phi_N} d^*(E \cap (E-p_1(n)) \cap \dots \cap (E-p_k(n))) \geq d^*(E)^{k+1}.\]
		Finally, apply \cite[Lemma~1.9]{abb}.
	\end{proof}
	In particular, we can apply Corollary \ref{largesetspoly} to obtain the following consequence.
	\begin{corollary}
		Let $F$ be a countable field with characteristic zero. Let $k \in \mathbb{N}$. Let $p_1,\dots,p_k \in F[x]$ be a family of independent polynomials and suppose that we have colored $F$ into $r$ distinct colors: $F=\bigsqcup_{i=1}^r C_i$ (here the union is disjoint). Then, there exists a color $i_0$ such that the set $\{ n \in F : C_{i_0} \cap (C_{i_0}-p_1(n)) \cap \dots \cap (C_{i_0}-p_k(n)) \neq \emptyset\}$ is syndetic.
	\end{corollary}
	We also have a result in topological dynamics.

\vspace{4pt}
	\begin{correp}[\ref{cor: findep}] Let $F$ be a countable field with characteristic zero. Let $(X,d,\break (T_n)_{n \in F})$ be a minimal topological dynamical system. Let $k \in \mathbb{N}$ and $p_1,\dots,p_k \in F[x]$ be a family of independent polynomials. Then there exists a dense $G_{\delta}$ set $X_0 \subseteq X$ such that for every $x \in X_0$ we have
		\begin{equation}\label{closure} \overline{\{(T_{p_1(n)}x,\dots,T_{p_k(n)}x) : n \in F\}}=\underbrace{X \times \dots \times X}_{k\ \mathrm{times}}.
		\end{equation}
	\end{correp}
	\begin{proof}
		Let $\mu$ be an ergodic $(T_n)_{n \in F}$-invariant measure on $X$. Since $(T_n)_{n \in F}$ is minimal, it follows that $\mu(U)>0$ for any non-empty open subset of $X$. Thus, for any non-empty open sets $U,V_1,\dots,V_k$, Theorem \ref{findep} implies that the set
		\begin{equation}\label{dense} \{ n \in F : U \cap T_{-p_1(n)}V_1\cap\dots\cap T_{-p_k(n)}V_k \neq \emptyset \}
		\end{equation}
		is infinite (even syndetic). To finish, we adapt \cite[Lemma~2.4]{huangshaoye} to our setting. Let $\mathcal{F}$ be a countable basis for the compact metric space $(X,d)$, and let
		\[ X_0:=\bigcap_{V_1,\dots,V_k \in \mathcal{F}} \bigcup_{n \in F} (T_{-p_1(n)}V_1 \cap \dots \cap T_{-p_k(n)}V_k).\]
		We have that $X_0$ is a countable intersection of open sets which are dense by \ref{dense}, so by the Baire category theorem, it is a dense $G_{\delta}$ set of points of $X$, each of which satisfies \ref{closure}.
	\end{proof}
	Using the fact that any topological dynamical system contains a minimal subsytem (which follows from Zorn's lemma), we see that Corollary~\ref{cor: findep} readily implies the following.
	\begin{corollary}
		Let $F$ be a countable field with characteristic zero. Let $(X,d,\break (T_n)_{n \in F})$ be a topological dynamical system, where $(X,d)$ is a compact metric space. Let $k \in \mathbb{N}$ and $p_1,\dots,p_k \in F[x]$ be a family of independent polynomials. Then, there exists a non-empty subset $F \subseteq X$ such that for each $x \in F$ we have
		\[ \overline{\{(T_{p_1(n)}x,\dots,T_{p_k(n)}x) : n \in F\}}=\overline{\{ T_nx : n \in F\}} \times \dots \times \overline{\{T_nx : n \in F\}}\]
	\end{corollary}
	\begin{remark}
		With some minor modifications, the results of this section can be extended to locally compact second-countable fields $F$ with characteristic zero, such as $\mathbb{R}$ or $\mathbb{C}$. We present the case of countable $F$ with characteristic zero for the sake of simplicity in exposition.
	\end{remark}
	\section{Joint ergodicity for totally ergodic actions of certain rings}\label{sec: 5}
	
	In this section, we sketch the proof of Theorem~\ref{rindep}. First, we recall some relevant definitions.

	\begin{definition}
		Let $R$ be a commutative ring. We say that $R$ is \emph{good} if it is a countable integral domain with characteristic zero, and every non-zero ideal has finite index in $R$.
	\end{definition}
	For good rings, we define total ergodicity in the following manner.
	\begin{definition}
		Let $R$ be a good ring and let $(X,\mathcal{B},\mu,(T_r)_{r \in R})$ be a measure preserving system. We say that $(T_r)_{r \in R}$ is \emph{totally ergodic} if, for every finite index subgroup $J \subseteq (R,+)$, the action $(T_n)_{n \in J}$ is ergodic.
	\end{definition}
	Let $R$ be a good ring. If $(T_r)_{r \in R}$ is a totally ergodic action on a probability space $(X,\mathcal{B},\mu)$, then, in particular, for every F\o lner sequence $(\Phi_N)$ in $R$, every $r \in R \setminus \{0\}$, and every function $f \in L^2(\mu)$, one has
	\[ \lim_{N \to \infty} \mathbb{E}_{n \in \Phi_N} T_{rn}f \ = \ \int_X f \ d\mu,\]
	where the convergence holds in $L^2(\mu)$. This observation justifies the following remark.
	\begin{remark} Let $R$ be a good ring. If $(T_r)_{r \in R}$ is a totally ergodic action on a probability space $(X,\mathcal{B},\mu)$, and $f \in \mathcal{E}((T_r)_{r \in R})$, where $T_rf=\chi(r)f$, for some $\chi \in \hat{R}$, then
		\[ \chi(b \, \cdot) \not\equiv 1 \]
		for all $b \in R \setminus\{0\}$. We will refer to characters  $\chi \in \hat{R}$ with this property as ``irrational characters''.
	\end{remark}
	
	For convenience, we restate Theorem~\ref{rindep} below.
	
\vspace{4pt}
	\begin{thmrep}[\ref{rindep}]
		Let $R$ be a good ring. Let $(X,\mu,(T_r)_{r \in R})$ be a totally ergodic $R$-system. Let $(\Phi_N)$ be a F\o lner sequence in $(R,+)$, let $k \in \mathbb{N}$, let $p_1,\dots,p_k \in R[x]$ be independent polynomials, and let $f_1,\dots,f_k \in L^{\infty}(\mu)$ be functions. Then
		\[ \lim_{N \to \infty} \frac{1}{|\Phi_N|}\sum_{n\in \Phi_N} T_{p_1(n)}f_1\cdots T_{p_k(n)}f_k=\prod_{j=1}^k \int_X f_i \ d\mu\]
		in $L^2(\mu)$.
	\end{thmrep}
	

	\begin{remark}
		Let $K$ be a number field. It is classical that its ring of integers $\mathcal{O}_K$ satisfies the definition of a good ring. Thus, Theorem~\ref{rindep} recovers, without the use of structure theory, \cite[Theorem~3.1]{ab} as a special case (see also item 2 in their abstract).
	\end{remark}
	
	We will now provide a sketch of the proof of Theorem~\ref{rindep}. Once again, we use Theorem~\ref{framainthm}, which reduces matters to proving an equidistribution result and establishing relevant seminorm estimates.
	
	First, we show the equidistribution result.
	
	\begin{theorem}\label{characterequidistributionrings}
		Let $R$ be a good ring. Let $\chi_1,\dots,\chi_k \in \hat{R}$ be irrational characters, not all trivial. Let $(\Phi_N)$ be a F\o lner sequence in $R$ and let $p_1,\dots,p_k \in R[x]$ be linearly independent polynomials with $p_i(0)=0$, $i=1,\dots,k$. Then
		\begin{equation}\label{qcharacteravgrings}
			\lim_{N \to \infty} \mathbb{E}_{n \in \Phi_N} \chi_1(p_1(n))\cdots \chi_k(p_k(n))=0.
		\end{equation}
	\end{theorem}
	\begin{proof}
		Assume without loss of generality that $\deg(p_i) \leq \deg(p_{i+1})$, $i=1,\dots,k-1$, and that all involved characters are non-trivial (or we delete them beforehand). Let $t \in \{1,\dots,k\}$ be such that $p_t, p_{t+1},\dots,p_k$ are all of the highest degree in the family $p_1,\dots,p_k$. We will reduce matters to the case where $\chi_t,\dots,\chi_k$ have no non-trivial relations; namely, for any $b_t,\dots,b_k \in R$, not all zero, it is not the case that
		\begin{equation}\label{characterrelation} \chi_t(b_t \, \cdot) \cdots \chi_k(b_k \, \cdot) \equiv 1.
		\end{equation}
		Indeed, assume that this was not the case, and thus, that we could find $b_t,\dots,b_k \in R$, not all zero, such that \ref{characterrelation} holds. Without loss of generality, $b_k \neq 0$. Since $R$ is a good ring, the ideal $(b_k)$ has finite index in $R$; let $r_1,\dots,r_s \in R$ be such that $\bigcup_{i=1}^s(b_k)+r_i=R$. We claim that, in order to prove that \ref{qcharacteravgrings} holds, it is enough to show that for any F\o lner sequences $(\Psi_N(i))$, $i=1,\ldots,s$, we have
		\[\lim_{N \to \infty} \sum_{i=1}^s\mathbb{E}_{n \in \Psi_N(i)} \chi_1(p_1(b_kn+r_i))\cdots \chi_k(p_k(b_kn+r_i))=0.\]
		Indeed, for each $N \in \mathbb{N}$, putting $A_N(i):=\{x \in R : b_kx \in \Phi_N - r_i\}$, we can write $\Phi_N$ as the disjoint union $\bigcup_{i=1}^sb_kA_N(i)+r_i$. (Note that $A_N(i)=\frac{(\Phi_N-r_i) \cap J}{b_k}$, where $J$ denotes the ideal $(b_k)$.) 
		Then, for any bounded function $f: R \to \mathbb{C}$, we have
		\[ \mathbb{E}_{n \in \Phi_N}f(n)=s\sum_{i=1}^s \mathbb{E}_{n \in A_N(i)} f(b_kn+r_i)+o_N(1),\]
		provided that $\frac{|A_N(i)|}{s|\Phi_N|} \xrightarrow[N \to \infty]{} 1$ for each $i \in \{1,\dots,s\}$.
		
		Consequently, we need only check that, for each $i \in \{1,\ldots, s\}$, $(A_N(i))$ is a F\o lner sequence with $\frac{|A_N(i)|}{s|\Phi_N|} \xrightarrow[N \to \infty]{} 1$. Let $t \in R$ and note that
		\begin{equation}\label{anisfolner} \frac{\left| A_N(i) \cap (A_N(i)+t) \right|}{|A_N(i)|}= \frac{\left|(\Phi_N-r_i) \cap (\Phi_N-r_i+b_kt) \cap J \right|}{|(\Phi_N-r_i) \cap J|}
		\end{equation}
		since $R$ is an integral domain.
		
		Now, since $(\Phi_N)$ is a F\o lner sequence and $[R : J]=s<\infty$, we can ``cancel'' the intersections with $J$ appearing on the right-hand side of \ref{anisfolner}. Indeed, notice that
		\[
\begin{split}
|\Phi_N|&=\sum_{i=1}^s |\Phi_N \cap (J+r_i)|=\sum_{i=1}^s |(\Phi_N-r_i) \cap J|\\& = \sum_{i=1}^s |\Phi_N \cap J| + o(|\Phi_N|) =s|\Phi_N \cap J| + o(|\Phi_N|).
\end{split}\]
		Using this fact and the F\o lner nature of $(\Phi_N)$, we see that $(A_N(i))$ is a F\o lner sequence with the desired property.
		
		Thus, it suffices to show  that for each $i \in \{1,\dots,s\}$ and every F\o lner sequence $(\Psi_N)$ we have
		\[ \lim_{N \to \infty} \mathbb{E}_{n \in \Psi_N} \chi_1(p_{1,i}(n))\cdots \chi_k(b_kp_{k,i}(n))=0,\]
		where $p_{j,i}(n)=p_{j}(b_kn+r_i)-p_j(r_i)$, $j=1,\dots,k-1$, and $p_{k,i}(n) =\frac{p_{k}(b_kn+r_i)-p_k(r_i)}{b_k}$. It is clear that, for each $i$, the new family of polynomials $p_{j,i}$ is linearly independent and is such that $p_{j,i}(0)=0$. Now, $\chi_k(b_kp_{k,i}(n))=\chi_t(-b_t p_{k,i}(n)) \cdots \chi_{k-1}(-b_{k-1}\break p_{k,i}(n))$, so we can delete the character $\chi_k$ from the averages under consideration.
		
		It is important to note that this process cannot delete all characters associated to polynomials with the maximum degree. Indeed, any irrational $\chi \in \hat{R}$ is such that
		\[ \chi(b\, \cdot) \not\equiv 1 \]
		for all $b \in R \setminus \{0\}$, so even if we could reduce it to only one non-trivial irrational character, the fact that such characters are irrational means that this process will not delete it. We conclude by applying the same reasoning as in the proof of Theorem~\ref{characterequidistribution}, since it does not make use of the fact that $F$ is a field.
	\end{proof}
	Second, we prove that families of independent polynomials over good rings are good for seminorm estimates. It is easy to see that most of the results in Section~\ref{sec: 4} hold for good rings as well. Here are the only two lemmas that require serious modification:
	\begin{lemma}\label{easy linear bound rings} Let $R$ be a good ring. Let $d \in \mathbb{N}$. Let $p : R^d \to R$ be a degree one polynomial. Then there exists a constant $c>0$ such that any $f \in L^\infty(X)$, we have
		\begin{equation*} \lim_{N\to\infty} \norm{\mathbb{E}_{\mathbf{g}\in\Phi_N} T_{p(\mathbf{g})}f}_{L^2(X)} \leq c\nnorm{f}_2
		\end{equation*}
		for any F\o lner sequence $(\Phi_N)$ in $R^d$.
	\end{lemma}
	\begin{proof} In coordinates, we write $p(\mathbf{g}) = p(g_1,\ldots, g_d) = a_1g_1 + \cdots + a_dg_d + a_0$ for $a_0, a_1, \ldots, a_d \in R$. Since $\nnorm{T_{a_0}f}_2 = \nnorm{f}_2$, we may replace $f$ by $T_{a_0}f$; thus assume $a_0 = 0$. 
		Since $R$ is a good ring, $R^d$ is a countable amenable group under addition. Since $p$ is linear, we may define a group action of $(R^d,+)$ on $L^2(X)$ by $U_\mathbf{g}(f') := T_{p(\mathbf{g})}f'$ for each $f' \in L^2(X)$ and $\mathbf{g} \in R^d$. By the mean ergodic theorem (Theorem~\ref{meanergthm}),
		\begin{equation*}
			\lim_{N\to\infty} \mathbb{E}_{\mathbf{g}\in\Phi_N} T_{p(\mathbf{g})}f = P_1f
		\end{equation*}
		in norm, where $P_1$ projects onto $\{ f' : T_{a_1g_1 + \cdots + a_dg_d}f'=f' \text{ for all } g_1,\ldots,g_d \in R\}$, which we observe equals $\{ f' : T_g f' = f' \text{ for all } g \in J\}$, where $J$ is the ideal generated by $a_1,\dots,a_d$. Thus, we have by the mean ergodic theorem that, for any F\o lner sequence $(\Psi_N)$ in $R$,
		\begin{align*}
			\lim_{N\to\infty} \norm{\mathbb{E}_{\mathbf{g}\in\Phi_N} T_{p(\mathbf{g})}f }_{L^2(X)}^2 \ & = \ \norm{P_1f}_{L^2(X)}^2 \\
			& = \ \left\langle P_1f,P_1f \right\rangle \\
			& = \ \left\langle f,P_1f \right\rangle \\
			& = \  \lim_{N\to\infty} \mathbb{E}_{g\in\Psi_N \cap J} \int_X f \cdot T_g \bar{f} \ d\mu  \\
			& = \ \lim_{N \to \infty}\mathbb{E}_{g \in \Psi_N \cap J}\nnorm{\Delta_g f}_0 \\
			\hspace{-.2in} & \leq \  \lim_{N \to \infty} \frac{|\Psi_N|}{|\Psi_N \cap J|} \cdot \mathbb{E}_{g \in \Psi_N}\nnorm{\Delta_g f}_1\\
			& \leq \ [R : J] \cdot \left( \mathbb{E}_{g \in \Psi_N}\nnorm{\Delta_g f}_1^2\right)^{1/2} \\
			& = \ [R : J]\cdot\nnorm{f}_2,
		\end{align*}
		proving the first lemma.
	\end{proof}
	\begin{lemma}\label{linear seminorm identity rings} Let $R$ be a good ring. Let $d \in \mathbb{N}$. Let $p : R^d \to R$ be a degree one polynomial with $p(\mathbf{0}) = 0$. Then there exists $c>0$ such that for any $f \in L^\infty(X)$, for any integer $k \geq 0$, and any F\o lner sequence $(\Phi_N)$ in $R^d$, we have
		\begin{equation*}
			\lim_{N\to\infty} \mathbb{E}_{\mathbf{g}\in\Phi_N} \nnorm{f \cdot T_{p(\mathbf{g})} \bar{f}}_k^{2^k} \leq c \nnorm{f}_{k+1}^{2^{k+1}}.
		\end{equation*}
	\end{lemma}
	\begin{proof} The main idea in this lemma is the same as in Lemma \ref{easy linear bound rings}. Unfolding definitions, letting $(\Psi_N)$ be a F\o lner sequence in $R$, and agreeing to discuss the marked equality afterwards, we calculate
		\begin{align*}
			& \lim_{N\to\infty} \mathbb{E}_{\mathbf{g}\in\Phi_N} \nnorm{f \cdot T_{p(\mathbf{g})} \bar{f}}_k^{2^k} \\
			& = \ \lim_{N\to\infty} \mathbb{E}_{\mathbf{g}\in\Phi_N} \int_{X^{[k]}} \bigotimes_{\varepsilon \in \{0,1\}^k} \mathcal{C}^{|\varepsilon|}(f\cdot T_{p(\mathbf{g})} \bar{f}) \ d\mu^{[k]} \\
			& = \ \lim_{N\to\infty} \mathbb{E}_{\mathbf{g}\in\Phi_N} \int_{X^{[k]}} \left( \bigotimes_{\varepsilon \in \{0,1\}^k} \mathcal{C}^{|\varepsilon|}f \right) \cdot \left( \bigotimes_{\varepsilon \in \{0,1\}^k} T_{p(\mathbf{g})}\mathcal{C}^{|\varepsilon|}\bar{f} \right) d\mu^{[k]} \\
			& = \ \lim_{N\to\infty} \mathbb{E}_{\mathbf{g}\in\Phi_N} \int_{X^{[k]}} \left( \bigotimes_{\varepsilon \in \{0,1\}^k} \mathcal{C}^{|\varepsilon|}f \right) \cdot T_{p(\mathbf{g})}^{[k]} \left( \bigotimes_{\varepsilon \in \{0,1\}^k} \mathcal{C}^{|\varepsilon|}\bar{f} \right) d\mu^{[k]} \\
			& \overset{*}{=} \ \lim_{N\to\infty} \mathbb{E}_{g \in \Psi_N \cap J} \int_{X^{[k]}}\left( \bigotimes_{\varepsilon \in \{0,1\}^k} \mathcal{C}^{|\varepsilon|}f \right) \cdot T_{g}^{[k]} \left( \bigotimes_{\varepsilon \in \{0,1\}^k} \mathcal{C}^{|\varepsilon|}\bar{f} \right) d\mu^{[k]} \\
			& = \ \lim_{N\to\infty} \mathbb{E}_{g\in\Psi_N \cap J} \nnorm{f \cdot T_g \bar{f}}_k^{2^k} \\
			& \leq \ \lim_{N\to\infty}\frac{|\Psi_N|}{|\Psi_N \cap J|}\cdot  \mathbb{E}_{g\in\Psi_N} \nnorm{f \cdot T_g \bar{f}}_k^{2^k} \\
			& = \ [R : J]\cdot \nnorm{f}_{k+1}^{2^{k+1}}.
		\end{align*}
		For the marked equality, consider the following. Since $p$ is linear, we may define a group action of $(R^d, +)$ on $L^2(X^{[k]})$ by $U_\mathbf{g}(f') := T_{p(\mathbf{g})}^{[k]}f'$ for each $f' \in L^2(X^{[k]})$ and $\mathbf{g} \in R^d$. By the mean ergodic theorem (Theorem~\ref{meanergthm}),
		\begin{multline}\label{one rings}
			\lim_{N\to\infty} \mathbb{E}_{\mathbf{g}\in\Phi_N} \int_{X^{[k]}} \left( \bigotimes_{\varepsilon \in \{0,1\}^k}\mathcal{C}^{|\varepsilon|} f \right) \cdot T_{p(\mathbf{g})}^{[k]} \left( \bigotimes_{\varepsilon \in \{0,1\}^k} \mathcal{C}^{|\varepsilon|}\bar{f} \right) d\mu^{[k]} = \\ \int_{X^{[k]}} \left( \bigotimes_{\varepsilon \in \{0,1\}^k} \mathcal{C}^{|\varepsilon|}f \right) \cdot P_1 \left( \bigotimes_{\varepsilon \in \{0,1\}^k} \mathcal{C}^{|\varepsilon|}\bar{f} \right) d\mu^{[k]},
		\end{multline}
		where $P_1$ is the orthogonal projection onto the subspace $\{ f' : T_{p(\mathbf{g})}^{[k]} f' = f' \text{ for all } \mathbf{g} \in R^d \}$, which we observe equals $\{ f' : T_g^{[k]} f' = f' \text{ for all } g \in J\}$, where $J$ is the ideal generated by the coefficients of $p$. Using the mean ergodic theorem for the action along $J$, equality \ref{one rings}, and the fact that a F\o lner sequence intersected with $J$ is a F\o lner sequence for $J$ completes the proof.
	\end{proof}
	We close this section with an example showcasing the necessity of independence of the polynomials $p_1,\dots,p_k$ in Theorem~\ref{rindep}. For the sake of concreteness, we will consider the case when $k=2$, $p_1(n) = n$, and $p_2(n) = an$ for some $a \in R \setminus \{0\}$, but it will be apparent from the method that we could extend this to an arbitrary family of polynomials which is subject to a non-trivial linear relation.
	\begin{remark}\label{example}
		Let $R$ be a countable ring, and let $X:=bR$ be the Bohr compactification\footnote{Given an abelian topological group $(G,+)$, the Bohr compactification $bG$ of $G$ is the maximal compact Hausdorff abelian group that is a compactification of $G$ and is such that any continuous homomorphism $\varphi : G \to K$ from $G$ to a compact Hausdorff group $K$ admits a unique extension from $bG$ to $K$. (These properties characterize $bG$ up to isomorphism.)} of $R$. Consider the measure-preserving system $(X,\mathrm{Borel}(X),(T_n)_{n \in R}, \mu)$, where $T_nx=x+n$ for any $x \in X$ (recall that $R$ is densely embedded in $X$), and $\mu$ is the unique Haar probability measure on $X$. Let $a \in R \setminus \{0\}$. Let $\chi \in \hat{R}$ be non-trivial, and consider $f:X \to \mathbb{S}^1$ defined as the unique extension of $\chi(a\, \cdot) : F \to \mathbb{S}^1$, and similarly $g: X \to \mathbb{S}^1$ extends $\bar{\chi}: F \to \mathbb{S}^1$. Now let $(\Phi_N)$ be a F\o lner sequence in $F$. Notice that for each $N \in \mathbb{N}$ and for every $x \in X$ we have
		\[ \mathbb{E}_{n \in \Phi_N} f(T_nx)g(T_{an}x)=f(x)g(x),\]
		as the equality is true if $x \in R$, which is a dense subset of $X$. On the other hand, since $f$ is an extension of a non-trivial character to $X$, it is actually the case that $f \in \hat{X}$, and $f \neq 1$. This implies that
		\[ \int_X f \ d\mu=0,\]
		but of course $|f(x)g(x)|=1$ for all $x \in X$ by construction, which shows that we do not have convergence to $\int_X f \ d\mu \cdot\int_X g \ d\mu$ if $p_1(n)=n$ and $p_2(n)=an$.
	\end{remark}

	\section*{Acknowledgments} We would like to thank Nikos Frantzikinakis for many helpful discussions and comments, including a technical correction to the definition of ``good for seminorm estimates'' that we made after publication; this correction does not affect the validity of any applications in the published version of the article. We would also like to thank Ethan Ackelsberg for a correction to the formulation of Theorem~\ref{findep}.

	\medskip
Received October 2021; revised January 2022; early access February 2022.
\medskip
	
\end{document}